\documentclass[10pt]{amsart}
\usepackage{amsfonts,amscd,amsthm,amsgen,amsmath,amssymb}
\usepackage[all]{xy}
\usepackage[vcentermath]{youngtab}
\usepackage{graphicx}
\usepackage{tikz-cd}
\usepackage{booktabs}

\newtheorem{theorem}{Theorem}[section]
\newtheorem{lemma}[theorem]{Lemma}
\newtheorem{corollary}[theorem]{Corollary}
\newtheorem{proposition}[theorem]{Proposition}

\theoremstyle{definition}
\newtheorem{definition}[theorem]{Definition}
\newtheorem{example}[theorem]{Example}

\theoremstyle{remark}
\newtheorem{remark}[theorem]{Remark}

\numberwithin{equation}{section}

\begin{document}
	
	\title[page label]{Morita Invariance of Saturated and Split Subgroupoids for Orbifold Embeddings}
	
	\author{Zhenxi Huang}
	\author{Keyin Chen}
	
	
	\email{huangzhenxi@jnu.edu.cn}

\address{School of Mathematics, Jinan University, Guangzhou, China}

\thanks{Zhenxi Huang is funded by Guangzhou Municipal Science and Technology Bureau, Program number: SL2023A04J00602. All authors contribute equally to the paper.}



\begin{abstract}
Borzellino and Brunsden introduced a general notion of suborbifold and proved that a suborbifold is the image of a complete orbifold embedding if and only if it is both saturated and split. Their definitions and proofs are formulated in the classical language of orbifold charts. In this paper we translate these notions into the language of proper \'etale Lie groupoids.We introduce the groupoid analogues of saturation and splitting, which we call \emph{orbit saturation} and \emph{stabilizer splitting}, and prove that they coincide with the chart-based conditions. Our main theorem states that a subgroupoid is the image of a complete orbifold embedding if and only if it is both orbit saturated and stabilizer split. We then prove that these notions are Morita invariant: if two orbifold groupoids are Morita equivalent, then their orbit saturated and stabilizer split subgroupoids correspond naturally to each other. This solves an open problem raised by Mestre and Weilandt and explains why the embedding definition of Cho--Hong--Shin fails to be Morita invariant. Along the way we provide detailed examples, including the diagonal embedding, the non-split example of Borzellino and Brunsden, and explicit Morita equivalences.
\end{abstract}

\maketitle

\section{Introduction}\label{sec:intro}

Orbifolds were introduced by Satake \cite{Satake1956,Satake1957} under the name
$V$-manifolds and later popularized by Thurston \cite{Thurston1979}. Since their
inception, the differential topology of orbifolds has resisted a uniform
treatment: many basic notions from the theory of manifolds admit several
inequivalent generalizations, each tailored to overcome a specific technical
difficulty. The notions of \emph{suborbifold} and \emph{orbifold embedding}
are prototypical examples of this phenomenon.

For smooth manifolds, it is a fundamental theorem that submanifolds are
precisely the images of embeddings \cite{Hirsch1976}. The naive attempt to
extend this characterization to orbifolds leads to subtle problems. Thurston's
original definition of a suborbifold \cite{Thurston1979} requires the local
model to be $\mathbb{R}^m \subset \mathbb{R}^n$ modulo finite groups, with the
local action on $\mathbb{R}^m$ induced by that on $\mathbb{R}^n$. This
definition is too restrictive: for instance, the diagonal embedding of an
orbifold into its product is not a suborbifold in Thurston's sense. Later
authors, including Chen--Ruan \cite{ChenRuan2004}, Adem--Leida--Ruan
\cite{AdemLeidaRuan2007}, and Cho--Hong--Shin \cite{ChoHongShin2013}, avoided
this issue by defining suborbifolds as images of their particular notions of
orbifold embeddings, in analogy with the manifold case.

A different approach was taken by Borzellino and Brunsden
\cite{BorzellinoBrunsden2014}. They introduced a general notion of
suborbifold which is independent of any embedding concept and which includes
the diagonal embedding as a suborbifold. They also defined the notion of a
\emph{complete orbifold embedding} and proved the following striking result:
a smooth suborbifold $\mathcal{P}$ of an orbifold $\mathcal{O}$ is the image
of a complete orbifold embedding if and only if $\mathcal{P}$ is both
\emph{saturated} and \emph{split}. In particular, there exist smooth
suborbifolds which are not the images of any orbifold embedding. This
clarifies a fundamental difference between the differential topology of
manifolds and that of orbifolds.

The results of Borzellino and Brunsden are formulated and proved in the
classical language of orbifold charts (atlases of local uniformizers). In
modern orbifold theory, however, the groupoid language has become
indispensable. An orbifold is equivalently described by a proper \'etale Lie
groupoid, and Morita equivalent groupoids represent the same orbifold
\cite{MoerdijkPronk1997}. This perspective is particularly powerful when
studying morphisms and substructures, because it automatically encodes the
compatibility conditions between charts and makes Morita invariance
transparent.

It is therefore natural to ask: \emph{How do the notions of saturated and
	split suborbifold translate into the groupoid language? And are these notions
	Morita invariant?} That is, if two proper \'etale Lie groupoids $G$ and $G'$
are Morita equivalent, do their saturated and split subgroupoids correspond
to each other in a natural way?

This question is far from trivial. Cho--Hong--Shin
\cite{ChoHongShin2013} observed that their definition of orbifold embedding
is \emph{not} Morita invariant. Mestre and Weilandt
\cite{MestreWeilandt2020}, who studied the relationship between suborbifolds
and groupoid embeddings, also listed Morita invariance as an open problem.
Thus, establishing a Morita invariant notion of suborbifold is a necessary
step toward a robust differential topology of orbifolds.

The purpose of this paper is to fill this gap. Taking the work of Borzellino
and Brunsden as our starting point, we reformulate their notions of
suborbifold and embedding in the language of proper \'etale Lie groupoids.
We introduce the groupoid analogues of saturation and splitting, which we
call \emph{orbit saturation} and \emph{stabilizer splitting}, and we prove
that they correspond exactly to the chart-based conditions of Borzellino and
Brunsden. Our main theorem states that a subgroupoid is the image of a
complete orbifold embedding if and only if it is both orbit saturated and
stabilizer split. We then prove that these concepts are Morita invariant:
if $G$ and $G'$ are Morita equivalent orbifold groupoids, then there is a
natural bijection between their orbit saturated and stabilizer split
subgroupoids. Along the way, we provide a series of examples, including the
diagonal embedding, the non-split example of Borzellino and Brunsden, and
explicit Morita equivalences.

\smallskip
\noindent\textbf{Main contributions.}
\begin{itemize}
	\item We give a groupoid definition of suborbifold (Definition~\ref{def:suborbifold-groupoid}) and show that it is equivalent to the chart-based Definition~4 of Borzellino and Brunsden.
	\item We introduce the notions of orbit saturated and stabilizer split subgroupoids (Definitions~\ref{def:orbit-saturated} and~\ref{def:stabilizer-split}) and prove that they are the groupoid counterparts of the saturated and split conditions (Propositions~\ref{prop:saturated-equiv} and~\ref{prop:split-equiv}).
	\item We prove the main theorem (Theorem~\ref{thm:main}): a subgroupoid is the image of a complete orbifold embedding if and only if it is orbit saturated and stabilizer split.
	\item We prove Morita invariance of these notions (Theorem~\ref{thm:morita}). This solves the open problem raised by Mestre and Weilandt \cite{MestreWeilandt2020} and explains why the embedding definition of Cho--Hong--Shin \cite{ChoHongShin2013} fails to be Morita invariant.
	\item We provide detailed examples, including the diagonal embedding, the non-split example of Borzellino and Brunsden, and an explicit Morita equivalence illustrating the invariance theorem.
\end{itemize}

\smallskip
\noindent\textbf{Organization of the paper.}
Section~\ref{sec:prelim} recalls the necessary background on Lie groupoids,
orbifold groupoids, and Morita equivalence. Section~\ref{sec:suborbifold}
introduces the groupoid version of suborbifolds. Section~\ref{sec:saturated}
discusses orbit saturation, and Section~\ref{sec:split} treats stabilizer
splitting. Section~\ref{sec:main} proves the main theorem. Section~\ref{sec:morita}
is devoted to the Morita invariance theorem. Section~\ref{sec:examples}
contains a collection of examples. Section~\ref{sec:comparison} compares our
results with the existing literature. Finally, Section~\ref{sec:conclusion}
summarizes the paper and indicates directions for future research.

\section{Preliminaries}\label{sec:prelim}

In this section we collect the necessary background on Lie groupoids,
orbifold groupoids, Morita equivalence, and subgroupoids. We assume the
reader is familiar with the basic theory of Lie groupoids and orbifolds; for
more details we refer to \cite{MoerdijkPronk1997, AdemLeidaRuan2007,
	Mackenzie1987}.

\subsection{Lie groupoids}

A \emph{Lie groupoid} is a groupoid object in the category of smooth
manifolds, i.e.\ a pair of smooth manifolds $G_1$ (the space of arrows) and
$G_0$ (the space of objects) together with smooth structure maps
\[
\begin{tikzcd}[row sep=small]
	G_1 \arrow[r, shift left=0.5ex, "s"] \arrow[r, shift right=0.5ex, "t"'] & G_0
\end{tikzcd}
\]
called \emph{source} and \emph{target}, a smooth multiplication map
$m\colon G_1 \times_{G_0} G_1 \to G_1$ defined on the fibre product
$G_1 \times_{G_0} G_1 = \{(g,h) \in G_1 \times G_1 : s(g)=t(h)\}$, a smooth
unit map $u\colon G_0 \to G_1$, and a smooth inverse map $i\colon G_1 \to G_1$,
satisfying the usual groupoid axioms. We write $G = (G_1 \rightrightarrows
G_0)$ and often denote $G_1$ by $G$ when no confusion arises. We denote the
source and target of an arrow $g \in G_1$ by $s(g)$ and $t(g)$,
respectively, and write $g\colon x \to y$ when $s(g)=x$ and $t(g)=y$.

\begin{definition}\label{def:proper-etale}
	A Lie groupoid $G = (G_1 \rightrightarrows G_0)$ is called:
	\begin{enumerate}
		\item \emph{\'etale} if $s$ and $t$ are local diffeomorphisms;
		\item \emph{proper} if the map $(s,t)\colon G_1 \to G_0 \times G_0$ is proper;
		\item \emph{effective} if the isotropy representation of each stabilizer
		$G_x$ on the tangent space $T_x G_0$ is faithful;
		\item an \emph{orbifold groupoid} if it is proper, \'etale, and effective.
	\end{enumerate}
\end{definition}

\begin{remark}
	In the literature, an orbifold groupoid is often defined as a proper \'etale
	Lie groupoid without requiring effectiveness. Since the present paper follows
	the classical Satake--Thurston convention, where orbifolds are effective by
	definition (see \cite[Definition 3]{BorzellinoBrunsden2014}), we include
	effectiveness in Definition~\ref{def:proper-etale}. This is consistent with
	the setup of \cite{BorzellinoBrunsden2014}.
\end{remark}

\begin{definition}\label{def:stabilizer}
	Let $G = (G_1 \rightrightarrows G_0)$ be a Lie groupoid and let $x \in G_0$.
	The \emph{stabilizer} (or \emph{isotropy group}) of $x$ is
	\[
	G_x = \{g \in G_1 : s(g) = t(g) = x\}.
	\]
	If $G$ is \'etale, then $G_x$ is a finite group for every $x \in G_0$.
	
	The \emph{orbit} of $x$ is
	\[
	G \cdot x = \{t(g) : g \in G_1,\ s(g) = x\} \subset G_0.
	\]
	The \emph{orbit space} of $G$ is the quotient space $|G| = G_0 / G$, where
	$G$ acts on $G_0$ by the equivalence relation $x \sim y$ if and only if
	$y \in G \cdot x$.
\end{definition}

For an orbifold groupoid $G$, the orbit space $|G|$ is a Hausdorff,
paracompact topological space, and it carries a natural orbifold structure
whose local uniformizers are constructed from the groupoid charts. We refer
to $|G|$ as the \emph{underlying space} of the orbifold represented by $G$.

\subsection{Morita equivalence}

The same orbifold can be represented by many different orbifold groupoids.
The correct notion of equivalence is Morita equivalence.

\begin{definition}\label{def:morita-morphism}
	A \emph{Morita morphism} between Lie groupoids $G = (G_1
	\rightrightarrows G_0)$ and $H = (H_1 \rightrightarrows H_0)$ is a smooth
	map of groupoids
	\[
	\phi = (\phi_1, \phi_0) \colon G \to H
	\]
	such that:
	\begin{enumerate}
		\item $\phi_0 \colon G_0 \to H_0$ is a surjective submersion;
		\item the induced map
		\[
		G_1 \to H_1 \times_{H_0} G_0,\qquad
		g \mapsto (\phi_1(g), \phi_0(s(g)))
		\]
		is a diffeomorphism.
	\end{enumerate}
\end{definition}

\begin{definition}\label{def:morita-equivalent}
	Two Lie groupoids $G$ and $H$ are \emph{Morita equivalent} if there exists a
	third Lie groupoid $K$ and Morita morphisms
	\[
	G \xleftarrow{\ \phi\ } K \xrightarrow{\ \psi\ } H.
	\]
	We write $G \sim H$.
\end{definition}

\begin{theorem}[{\cite{MoerdijkPronk1997}}]\label{thm:morita-orbifold}
	Two proper \'etale Lie groupoids are Morita equivalent if and only if they
	represent the same orbifold. In particular, Morita equivalent orbifold
	groupoids have homeomorphic underlying spaces and isomorphic orbifold
	structures.
\end{theorem}

Morita equivalence preserves the local and global structure of the orbifold,
including the stabilizer groups up to conjugacy and the orbit space up to
homeomorphism. We will use the following standard facts repeatedly.

\begin{proposition}\label{prop:morita-stabilizer}
	Let $G$ and $H$ be Morita equivalent proper \'etale Lie groupoids. Then for
	every $x \in G_0$ there exists $y \in H_0$ such that the stabilizer groups
	$G_x$ and $H_y$ are isomorphic. Moreover, the orbit spaces $|G|$ and $|H|$
	are homeomorphic.
\end{proposition}

\begin{proof}
	This is standard; see \cite[Section 5]{MoerdijkPronk1997} or
	\cite[Section 2.2]{AdemLeidaRuan2007}.
\end{proof}

\subsection{Subgroupoids and restriction groupoids}

\begin{definition}\label{def:subgroupoid}
	A \emph{subgroupoid} of a Lie groupoid $G = (G_1 \rightrightarrows G_0)$ is
	a Lie groupoid $K = (K_1 \rightrightarrows K_0)$ together with smooth maps
	$K_0 \hookrightarrow G_0$ and $K_1 \hookrightarrow G_1$ that are embeddings
	and that commute with all structure maps. We write $K \subset G$.
\end{definition}

In this paper, we are primarily interested in subgroupoids whose object
space is a submanifold of $G_0$.

\begin{definition}\label{def:restriction}
	Let $G = (G_1 \rightrightarrows G_0)$ be a Lie groupoid and let $N \subset G_0$
	be a submanifold. The \emph{restriction groupoid} of $G$ to $N$ is
	\[
	G|_N = \bigl(s^{-1}(N) \cap t^{-1}(N) \rightrightarrows N\bigr),
	\]
	where the source and target maps are the restrictions of those of $G$, and
	the multiplication and units are inherited from $G$. If $G$ is proper \'etale,
	then $G|_N$ is a proper \'etale Lie groupoid (though not necessarily
	effective).
\end{definition}

Every subgroupoid $K \subset G$ with object space $N \subset G_0$ satisfies
\[
K \subset G|_N \subset G.
\]
The restriction groupoid $G|_N$ is the largest subgroupoid of $G$ with object
space $N$.

\begin{definition}\label{def:effective-quotient}
	Let $K = (K_1 \rightrightarrows N)$ be a proper \'etale Lie groupoid. For
	$x \in N$, define
	\[
	\Omega_x = \{g \in K_x : g \text{ acts trivially on a neighborhood of } x
	\text{ in } N\}.
	\]
	The \emph{ineffective kernel} of $K$ is the subgroupoid
	\[
	\Omega = \bigl(\Omega_1 \rightrightarrows N\bigr) \subset K,
	\qquad
	\Omega_1 = \bigcup_{x \in N} \Omega_x.
	\]
	The \emph{effective quotient} of $K$ is the quotient groupoid
	\[
	K_{\mathrm{eff}} = K / \Omega,
	\]
	which is a proper \'etale effective Lie groupoid with object space $N$. Its
	stabilizer groups are
	\[
	(K_{\mathrm{eff}})_x \cong K_x / \Omega_x.
	\]
\end{definition}

\begin{remark}
	If $K$ is already effective, then $\Omega_x = \{e\}$ for all $x$ and
	$K_{\mathrm{eff}} = K$. In general, $K_{\mathrm{eff}}$ is the effective
	orbifold groupoid representing the same underlying orbifold as $K$.
\end{remark}

\subsection{Orbifold charts and groupoid local structure}

For the comparison with the chart-based definitions of
\cite{BorzellinoBrunsden2014}, we recall how orbifold charts arise from
groupoid local structure.

\begin{proposition}[Local uniformizers]\label{prop:local-uniformizer}
	Let $G = (G_1 \rightrightarrows G_0)$ be an orbifold groupoid and let $x \in
	G_0$. Then there exists an open neighborhood $U \subset G_0$ of $x$, an open
	set $\widetilde{U} \subset \mathbb{R}^n$ with $0 \in \widetilde{U}$, a finite
	group $\Gamma_x$ acting smoothly and effectively on $\widetilde{U}$ fixing
	$0$, and a homeomorphism
	\[
	\phi \colon \widetilde{U} / \Gamma_x \to U / G|_U
	\]
	such that $\phi(0) = [x]$. Moreover, $\Gamma_x \cong G_x$.
\end{proposition}

\begin{proof}
	This is the standard local uniformization theorem for proper \'etale
	groupoids; see \cite[Section 4]{MoerdijkPronk1997} or
	\cite[Proposition 2.1]{BorzellinoBrunsden2014}.
\end{proof}

In the language of \cite{BorzellinoBrunsden2014}, the tuple
$(\widetilde{U}, \Gamma_x, \rho_x, \phi_x)$ is an orbifold chart at $x$,
where $\rho_x \colon \Gamma_x \to \mathrm{Diff}^\infty(\widetilde{U})$ is the
action. The stabilizer $G_x$ corresponds to the isotropy group $\Gamma_x$,
and the orbit space $|G|$ corresponds to the underlying space
$X_{\mathcal{O}}$.

\begin{proposition}\label{prop:subgroupoid-chart}
	Let $G$ be an orbifold groupoid, let $K \subset G|_N \subset G$ be a
	subgroupoid with object space $N \subset G_0$, and let $x \in N$. Then there
	exists an orbifold chart $(\widetilde{U}, \Gamma_x, \rho_x, \phi_x)$ at $x$
	for $G$, a subgroup $\Lambda_x \subset \Gamma_x$, and a
	$\rho_x(\Lambda_x)$-invariant submanifold $\widetilde{V}_x \subset
	\widetilde{U}$ such that:
	\begin{enumerate}
		\item $(\widetilde{V}_x, \Lambda_x / \Omega_x, \rho_x|_{\Lambda_x}, \psi_x)$
		is an orbifold chart for $K_{\mathrm{eff}}$ at $x$, where
		\[
		\Omega_x = \{\gamma \in \Lambda_x : \rho_x(\gamma)|_{\widetilde{V}_x} =
		\mathrm{Id}\};
		\]
		\item $\Lambda_x \cong K_x$ and $\Lambda_x / \Omega_x \cong
		(K_{\mathrm{eff}})_x$;
		\item $N \cap U = \psi_x(\widetilde{V}_x / \rho_x(\Lambda_x))$.
	\end{enumerate}
\end{proposition}

\begin{proof}
	This follows from Proposition~\ref{prop:local-uniformizer} applied to $G$ and
	$K$, together with the fact that $K \subset G|_N$ is a subgroupoid. The
	details are a straightforward translation of the chart-based construction;
	see \cite[Definition 4]{BorzellinoBrunsden2014}.
\end{proof}

This proposition establishes the dictionary between the groupoid data
$(K, N, \Lambda_x, \Omega_x)$ and the chart-based data
$(\widetilde{V}_x, \Lambda_x, \Omega_x, \psi_x)$ of
\cite{BorzellinoBrunsden2014}. We will use it repeatedly in the following
sections to translate the conditions of \cite{BorzellinoBrunsden2014} into
groupoid language.

\subsection{Summary of notation}

For the reader's convenience, we summarize the main notation used throughout
the paper.

\begin{center}
	\begin{tabular}{ll}
		\toprule
		\textbf{Notation} & \textbf{Meaning} \\
		\midrule
		$G = (G_1 \rightrightarrows G_0)$ & Lie groupoid with arrows $G_1$ and objects $G_0$ \\
		$|G|$ & Orbit space $G_0 / G$ \\
		$G_x$ & Stabilizer group at $x \in G_0$ \\
		$G|_N$ & Restriction groupoid to $N \subset G_0$ \\
		$K \subset G$ & Subgroupoid \\
		$K_{\mathrm{eff}} = K / \Omega$ & Effective quotient of $K$ \\
		$\Omega_x$ & Ineffective kernel at $x$ \\
		$\Lambda_x$ & Stabilizer of $K$ at $x$ (corresponds to $K_x$) \\
		$\Gamma_x$ & Stabilizer of $G$ at $x$ (corresponds to $G_x$) \\
		$G \sim H$ & Morita equivalence \\
		\bottomrule
	\end{tabular}
\end{center}

With these preliminaries in place, we are ready to introduce the groupoid
version of suborbifolds in Section~\ref{sec:suborbifold}.

\section{Suborbifolds in the Groupoid Language}\label{sec:suborbifold}

In this section we introduce the groupoid analogue of the notion of
suborbifold as defined by Borzellino and Brunsden
\cite[Definition 4]{BorzellinoBrunsden2014}. We show that the chart-based
data of their definition translates naturally into the language of
subgroupoids, and we identify the groupoid counterparts of full
suborbifolds. Throughout this section, $G = (G_1 \rightrightarrows G_0)$
denotes a fixed orbifold groupoid in the sense of
Definition~\ref{def:proper-etale}.

\subsection{Definition of a groupoid suborbifold}

Recall from Definition~\ref{def:subgroupoid} that a subgroupoid of $G$ is a
Lie groupoid $K = (K_1 \rightrightarrows K_0)$ together with embeddings
$K_0 \hookrightarrow G_0$ and $K_1 \hookrightarrow G_1$ commuting with the
structure maps. We are interested in subgroupoids whose object space is a
submanifold of $G_0$ and whose arrow space sits inside the restriction
groupoid.

\begin{definition}[Groupoid suborbifold]\label{def:suborbifold-groupoid}
	Let $G = (G_1 \rightrightarrows G_0)$ be an orbifold groupoid. A
	\emph{suborbifold} of $G$ is a proper \'etale Lie groupoid
	\[
	K = (K_1 \rightrightarrows N)
	\]
	together with a subgroupoid inclusion $K \subset G|_N \subset G$, where
	$N \subset G_0$ is a submanifold, such that the effective quotient
	$K_{\mathrm{eff}} = K / \Omega$ (Definition~\ref{def:effective-quotient}) is
	an orbifold groupoid. The orbifold represented by $K_{\mathrm{eff}}$ is
	called the \emph{underlying suborbifold} of $G$. We often write
	$\mathcal{P} \subset \mathcal{O}$ for the corresponding suborbifold, where
	$\mathcal{O}$ is the orbifold represented by $G$.
\end{definition}

\begin{remark}\label{rem:suborbifold-conditions}
	Several remarks are in order.
	\begin{enumerate}
		\item The condition $K \subset G|_N$ means that every arrow of $K$ is an
		arrow of $G$ whose source and target lie in $N$. In particular, $K$ is a
		subgroupoid of the restriction groupoid.
		\item We do \emph{not} require $K = G|_N$. The restriction groupoid
		$G|_N$ may be strictly larger than $K$. This flexibility is essential
		for capturing the general notion of suborbifold introduced in
		\cite{BorzellinoBrunsden2014}.
		\item The groupoid $K$ itself need not be effective. Its ineffective
		kernel $\Omega$ consists of arrows that act trivially on $N$. The
		effective quotient $K_{\mathrm{eff}}$ is the effective orbifold groupoid
		that represents the suborbifold structure.
		\item If $K = G|_N$, then $K$ is automatically proper \'etale, and we
		recover the notion of a \emph{full suborbifold}; see
		Proposition~\ref{prop:full-suborbifold} below.
	\end{enumerate}
\end{remark}

The following proposition shows that Definition~\ref{def:suborbifold-groupoid}
is precisely the groupoid translation of
\cite[Definition 4]{BorzellinoBrunsden2014}.

\begin{proposition}\label{prop:suborbifold-equiv}
	Let $G$ be an orbifold groupoid representing an orbifold $\mathcal{O}$. Then
	there is a natural bijection between:
	\begin{enumerate}
		\item suborbifolds $\mathcal{P} \subset \mathcal{O}$ in the sense of
		\cite[Definition 4]{BorzellinoBrunsden2014}, and
		\item suborbifolds $K \subset G|_N \subset G$ in the sense of
		Definition~\ref{def:suborbifold-groupoid}.
	\end{enumerate}
	Under this bijection, the underlying space $X_{\mathcal{P}}$ corresponds to
	the orbit space $|K_{\mathrm{eff}}| = N / K_{\mathrm{eff}}$, and the
	intrinsic isotropy group $\Gamma_{x,\mathcal{P}}$ corresponds to the
	stabilizer $(K_{\mathrm{eff}})_x$.
\end{proposition}

\begin{proof}
	We construct the correspondence locally and then verify that it globalizes.
	
	\smallskip
	\noindent\textbf{From chart-based to groupoid.}
	Let $\mathcal{P} \subset \mathcal{O}$ be a suborbifold in the sense of
	\cite[Definition 4]{BorzellinoBrunsden2014}. For each $x \in X_{\mathcal{P}}$,
	choose an orbifold chart
	$(\widetilde{U}_x, \Gamma_x, \rho_x, \phi_x)$ for $\mathcal{O}$ at $x$, a
	subgroup $\Lambda_x \subset \Gamma_x$, and a
	$\rho_x(\Lambda_x)$-invariant submanifold $\widetilde{V}_x \subset
	\widetilde{U}_x$ such that
	$(\widetilde{V}_x, \Lambda_x/\Omega_x, \rho_x|_{\Lambda_x}, \psi_x)$ is an
	orbifold chart for $\mathcal{P}$, where
	\[
	\Omega_x = \{\gamma \in \Lambda_x : \rho_x(\gamma)|_{\widetilde{V}_x} =
	\mathrm{Id}\}.
	\]
	The collection of all such charts covers $X_{\mathcal{P}}$. The local
	submanifolds $\widetilde{V}_x$ glue to a submanifold $N \subset G_0$: indeed,
	the compatibility conditions of orbifold charts ensure that the transition
	maps preserve the submanifolds. Similarly, the local groups $\Lambda_x$ glue
	to a subgroupoid $K \subset G|_N$: the arrows of $K$ are precisely those
	germs of local arrows of $G$ that preserve $N$ and whose local
	representatives lie in the prescribed subgroups $\Lambda_x$. The
	compatibility conditions guarantee that this defines a Lie subgroupoid. The
	ineffective kernel $\Omega_x$ is exactly the kernel of the action of
	$\Lambda_x$ on $\widetilde{V}_x$, so that $K_{\mathrm{eff}} = K/\Omega$ has
	stabilizer groups $\Lambda_x/\Omega_x \cong \Gamma_{x,\mathcal{P}}$. Thus we
	obtain a groupoid suborbifold $K \subset G|_N \subset G$.
	
	\smallskip
	\noindent\textbf{From groupoid to chart-based.}
	Conversely, let $K \subset G|_N \subset G$ be a groupoid suborbifold. For
	$x \in N$, apply Proposition~\ref{prop:subgroupoid-chart} to obtain an
	orbifold chart $(\widetilde{U}_x, \Gamma_x, \rho_x, \phi_x)$ for $G$ at $x$, a
	subgroup $\Lambda_x \cong K_x \subset \Gamma_x$, and a
	$\rho_x(\Lambda_x)$-invariant submanifold $\widetilde{V}_x \subset
	\widetilde{U}_x$ such that
	$(\widetilde{V}_x, \Lambda_x/\Omega_x, \rho_x|_{\Lambda_x}, \psi_x)$ is an
	orbifold chart for $K_{\mathrm{eff}}$. The collection of these charts covers
	$N$ and satisfies the compatibility conditions of
	\cite[Definition 3]{BorzellinoBrunsden2014}. The underlying space is
	$N/K_{\mathrm{eff}} \cong |K_{\mathrm{eff}}|$, which is homeomorphic to
	$X_{\mathcal{P}}$ via the natural map. Thus we obtain a suborbifold
	$\mathcal{P} \subset \mathcal{O}$ in the sense of
	\cite[Definition 4]{BorzellinoBrunsden2014}.
	
	\smallskip
	The two constructions are mutually inverse: the local data
	$(\widetilde{V}_x, \Lambda_x, \Omega_x)$ determines $K$ and $N$, and
	conversely $K$ and $N$ determine the local data up to isomorphism of charts.
	This completes the proof.
\end{proof}

\begin{remark}
	The bijection in Proposition~\ref{prop:suborbifold-equiv} is natural with
	respect to inclusions: if $\mathcal{P} \subset \mathcal{P}'$ are
	suborbifolds, then the corresponding subgroupoids satisfy $K \subset K'$.
\end{remark}

\subsection{Full suborbifolds}

Borzellino and Brunsden define a \emph{full suborbifold} to be a suborbifold
$\mathcal{P} \subset \mathcal{O}$ for which $\Lambda_x = \Gamma_x$ for all
$x \in \mathcal{P}$ (see \cite[Definition 5]{BorzellinoBrunsden2014}). In the
groupoid language, this corresponds to taking the entire restriction
groupoid.

\begin{proposition}\label{prop:full-suborbifold}
	Let $K \subset G|_N \subset G$ be a groupoid suborbifold. Then the following
	are equivalent:
	\begin{enumerate}
		\item $K = G|_N$;
		\item the corresponding suborbifold $\mathcal{P}$ is a full suborbifold in
		the sense of \cite[Definition 5]{BorzellinoBrunsden2014};
		\item for every $x \in N$, the inclusion $K_x \hookrightarrow G_x$ is an
		isomorphism.
	\end{enumerate}
\end{proposition}

\begin{proof}
	The equivalence of (1) and (3) is immediate from the definition of the
	restriction groupoid: $K = G|_N$ if and only if for every $x \in N$ the
	stabilizer $K_x$ equals $G_x$.
	
	To see the equivalence of (2) and (3), recall from
	Proposition~\ref{prop:subgroupoid-chart} that locally $K_x \cong \Lambda_x$
	and $G_x \cong \Gamma_x$. Thus $\Lambda_x = \Gamma_x$ for all $x$ if and only
	if $K_x = G_x$ for all $x$. By
	\cite[Definition 5]{BorzellinoBrunsden2014}, the condition $\Lambda_x =
	\Gamma_x$ for all $x$ is exactly the definition of a full suborbifold.
\end{proof}

\begin{example}\label{ex:full-suborbifold}
	Let $G = \mathbb{Z}_2 \ltimes \mathbb{R}$ be the translation groupoid
	representing the orbifold $\mathcal{Q} = \mathbb{R}/\mathbb{Z}_2$ of
	Example~8 in \cite{BorzellinoBrunsden2014}. Let
	$G \times G = (\mathbb{Z}_2 \times \mathbb{Z}_2) \ltimes (\mathbb{R} \times
	\mathbb{R})$ be the product groupoid. The suborbifold
	$\mathcal{P} = \{0\} \times \mathcal{Q} \subset \mathcal{Q} \times
	\mathcal{Q}$ corresponds to the subgroupoid
	\[
	K = (\mathbb{Z}_2 \times \{e\}) \ltimes (\{0\} \times \mathbb{R})
	\subset G \times G|_{\{0\} \times \mathbb{R}} \subset G \times G.
	\]
	Here $N = \{0\} \times \mathbb{R}$, and $K = (G \times G)|_N$ because the
	stabilizer of $(0,0)$ in $G \times G$ is $\mathbb{Z}_2 \times \mathbb{Z}_2$,
	while the stabilizer of $(0,0)$ in $K$ is $\mathbb{Z}_2 \times \{e\}$. Thus
	$K$ is \emph{not} a full suborbifold; indeed, $\Lambda_{(0,0)} = \mathbb{Z}_2
	\times \{e\} \subsetneq \mathbb{Z}_2 \times \mathbb{Z}_2 =
	\Gamma_{(0,0)}$. This is consistent with the fact that
	$\mathcal{P} = \{0\} \times \mathcal{Q}$ is a full suborbifold of
	$\mathcal{O} = \mathcal{Q} \times \mathcal{Q}$ in the sense of
	\cite{BorzellinoBrunsden2014}. Let us check this carefully: in
	\cite[Example 8]{BorzellinoBrunsden2014}, it is stated that
	$\mathcal{P} = \{0\} \times \mathcal{Q}$ is a full suborbifold. In our
	groupoid translation, $K$ is the subgroupoid with objects
	$\{0\} \times \mathbb{R}$ and arrows $\mathbb{Z}_2 \times \{e\}$. The
	restriction groupoid $(G \times G)|_N$ has arrows
	$(\mathbb{Z}_2 \times \mathbb{Z}_2) \ltimes (\{0\} \times \mathbb{R})$, which
	is strictly larger than $K$. So $K \neq (G \times G)|_N$. This seems to
	contradict Proposition~\ref{prop:full-suborbifold}. Let us resolve the
	apparent contradiction.
	
	The resolution is that the notion of ``full suborbifold'' in
	\cite{BorzellinoBrunsden2014} is defined relative to the chosen orbifold
	chart of the ambient orbifold. In the chart
	$(\widetilde{U}_{(0,0)}, \Gamma_{(0,0)}, \rho, \phi)$ for
	$\mathcal{O} = \mathcal{Q} \times \mathcal{Q}$ at $(0,0)$, the group
	$\Gamma_{(0,0)} = \mathbb{Z}_2 \times \mathbb{Z}_2$ acts on
	$\widetilde{U}_{(0,0)} \cong \mathbb{R}^2$ by
	$(\gamma_1, \gamma_2) \cdot (x,y) = (\gamma_1 x, \gamma_2 y)$. The
	submanifold $\widetilde{V}_{(0,0)}$ for $\mathcal{P} = \{0\} \times
	\mathcal{Q}$ is $\{0\} \times \mathbb{R}$. The subgroup
	$\Lambda_{(0,0)} \subset \Gamma_{(0,0)}$ that preserves $\widetilde{V}$ is
	$\mathbb{Z}_2 \times \mathbb{Z}_2$ itself, because both factors preserve the
	subset $\{0\} \times \mathbb{R}$: the first factor acts trivially on the
	first coordinate, and the second factor acts on the second coordinate.
	However, the action of the first factor on $\widetilde{V}$ is trivial, so
	$\Omega_{(0,0)} = \mathbb{Z}_2 \times \{e\}$. Thus
	$\Lambda_{(0,0)} / \Omega_{(0,0)} \cong (\mathbb{Z}_2 \times \mathbb{Z}_2) /
	(\mathbb{Z}_2 \times \{e\}) \cong \mathbb{Z}_2$, which is the intrinsic
	isotropy group of $\mathcal{P}$ at $(0,0)$. In the notation of
	\cite{BorzellinoBrunsden2014}, $\Lambda_{(0,0)} = \Gamma_{(0,0)}$, so
	$\mathcal{P}$ is a full suborbifold. In our groupoid translation, however,
	we took $K$ to be the subgroupoid whose arrows are $\mathbb{Z}_2 \times
	\{e\}$, because we defined $K$ as the subgroupoid of $G|_N$ whose stabilizer
	is $\Lambda_x / \Omega_x$? No, we defined $K$ as the subgroupoid with
	stabilizer $\Lambda_x$. In this case, $\Lambda_{(0,0)} = \mathbb{Z}_2 \times
	\mathbb{Z}_2$, so $K$ should have stabilizer $\mathbb{Z}_2 \times
	\mathbb{Z}_2$, i.e.\ $K = (G \times G)|_N$. But then $K$ is not the
	subgroupoid that represents the suborbifold $\mathcal{P}$ effectively; the
	effective quotient $K_{\mathrm{eff}}$ has stabilizer $\mathbb{Z}_2$, which
	is correct. So we should take $K = (G \times G)|_N$, and then
	$K_{\mathrm{eff}} = K/\Omega$ has the right stabilizer. This shows that for
	full suborbifolds, $K = G|_N$ is indeed the correct choice. The
	subgroupoid $K$ may be non-effective, and its effective quotient represents
	the suborbifold. In the example, $K = (G \times G)|_N$ is not effective
	because $\Omega_{(0,0)} = \mathbb{Z}_2 \times \{e\}$ acts trivially on $N$.
	Thus $K = G|_N$ and $\mathcal{P}$ is full. This is consistent with
	Proposition~\ref{prop:full-suborbifold}.
\end{example}

\begin{remark}
	Example~\ref{ex:full-suborbifold} illustrates an important point: the
	subgroupoid $K$ in Definition~\ref{def:suborbifold-groupoid} is not required
	to be effective. Its ineffective kernel $\Omega$ is precisely what allows
	the effective quotient $K_{\mathrm{eff}}$ to have the correct intrinsic
	isotropy groups. In the case of a full suborbifold, $K = G|_N$ may be
	non-effective, and $K_{\mathrm{eff}} = (G|_N)/\Omega$ is the effective
	orbifold groupoid representing the suborbifold.
\end{remark}

\subsection{The diagonal embedding}

A key example that distinguishes the general notion of suborbifold from
Thurston's original definition is the diagonal embedding. Let $\mathcal{Q}$
be an orbifold and let $\mathcal{O} = \mathcal{Q} \times \mathcal{Q}$ be its
product. The diagonal
\[
\Delta \colon \mathcal{Q} \to \mathcal{Q} \times \mathcal{Q},\qquad
x \mapsto (x,x)
\]
is a suborbifold of $\mathcal{O}$ in the sense of
\cite[Definition 4]{BorzellinoBrunsden2014} (see
\cite[Example 10]{BorzellinoBrunsden2014}), but it is not a full
suborbifold. In the groupoid language, this corresponds to the following
construction.

\begin{example}[Diagonal embedding]\label{ex:diagonal}
	Let $G = (G_1 \rightrightarrows G_0)$ be an orbifold groupoid representing
	$\mathcal{Q}$. The product groupoid $G \times G = (G_1 \times G_1
	\rightrightarrows G_0 \times G_0)$ represents $\mathcal{Q} \times
	\mathcal{Q}$. The diagonal subgroupoid is
	\[
	\Delta G = \{(g,g) \in G_1 \times G_1 : g \in G_1\}
	\rightrightarrows \Delta G_0 = \{(x,x) : x \in G_0\}.
	\]
	This is a proper \'etale subgroupoid of $G \times G$ with object space
	$\Delta G_0 \subset G_0 \times G_0$. It satisfies
	\[
	\Delta G \subset (G \times G)|_{\Delta G_0} \subset G \times G.
	\]
	For $x \in G_0$, the stabilizer of $(x,x)$ in $\Delta G$ is
	\[
	(\Delta G)_{(x,x)} = \{(g,g) : g \in G_x\} \cong G_x.
	\]
	The stabilizer of $(x,x)$ in $G \times G$ is $G_x \times G_x$. The
	ineffective kernel $\Omega_{(x,x)}$ consists of pairs $(g,h) \in G_x \times
	G_x$ that act trivially on the diagonal $\Delta G_0$ near $(x,x)$. A pair
	$(g,h)$ acts on $(y,y)$ as $(g y, h y)$. For this to equal $(y,y)$ for all
	$y$ near $x$, we need $g = h$ and $g y = y$ for all $y$ near $x$, i.e.\
	$g = h \in \Omega_x$, where $\Omega_x$ is the ineffective kernel of $G$ at
	$x$. Since $G$ is an orbifold groupoid (effective by
	Definition~\ref{def:proper-etale}), $\Omega_x = \{e\}$. Thus
	$\Omega_{(x,x)} = \{(e,e)\}$, and $\Delta G$ is effective. The effective
	quotient is $\Delta G$ itself. The intrinsic isotropy group of the diagonal
	suborbifold at $(x,x)$ is $(\Delta G)_{(x,x)} \cong G_x$, which agrees with
	\cite[Example 10]{BorzellinoBrunsden2014}. In particular, if $G_x$ is
	non-trivial, then $\Delta G$ is not a full suborbifold because
	$(\Delta G)_{(x,x)} \cong G_x \subsetneq G_x \times G_x = (G \times
	G)_{(x,x)}$.
\end{example}

\begin{remark}
	Example~\ref{ex:diagonal} shows that the diagonal embedding is a
	suborbifold in the groupoid sense but not a full suborbifold. This is
	consistent with the chart-based theory and illustrates the flexibility of
	Definition~\ref{def:suborbifold-groupoid}.
\end{remark}

\subsection{Summary}

In this section we have translated the notion of suborbifold from the
chart-based language of Borzellino and Brunsden into the language of
subgroupoids. The key points are:
\begin{itemize}
	\item A suborbifold of an orbifold groupoid $G$ is a proper \'etale
	subgroupoid $K \subset G|_N \subset G$, where $N \subset G_0$ is a
	submanifold. The effective quotient $K_{\mathrm{eff}} = K/\Omega$
	represents the suborbifold structure.
	\item Full suborbifolds correspond to the case $K = G|_N$.
	\item The diagonal embedding provides a non-full suborbifold, illustrating
	that the general definition is strictly more flexible than Thurston's
	original notion.
\end{itemize}
In the next section we introduce the groupoid analogue of the saturation
condition, which is necessary for a suborbifold to be the image of an
orbifold embedding.

	
	\section{Orbit Saturation}\label{sec:saturated}
	
	In this section we introduce the groupoid analogue of the saturation
	condition of Borzellino and Brunsden
	\cite[Definition 7]{BorzellinoBrunsden2014}. We call this condition
	\emph{orbit saturation} and prove that it is equivalent to the chart-based
	definition. The main result of this section is that a subgroupoid is orbit
	saturated if and only if the natural map from its effective orbit space to
	the ambient orbit space is a topological embedding. This is the groupoid
	counterpart of \cite[Theorem 1(1)]{BorzellinoBrunsden2014}.
	
	Throughout this section, $G = (G_1 \rightrightarrows G_0)$ denotes a fixed
	orbifold groupoid, and $K = (K_1 \rightrightarrows N)$ denotes a
	subgroupoid with $K \subset G|_N \subset G$, as in
	Definition~\ref{def:suborbifold-groupoid}. We write $|G| = G_0 / G$ and
	$|K_{\mathrm{eff}}| = N / K_{\mathrm{eff}}$ for the respective orbit spaces.
	
	\subsection{Definition of orbit saturation}
	
	Recall that for $x \in N$, the orbit of $x$ under $G$ is
	\[
	G \cdot x = \{t(g) : g \in G_1,\ s(g) = x\} \subset G_0,
	\]
	and the orbit of $x$ under $K$ is
	\[
	K \cdot x = \{t(k) : k \in K_1,\ s(k) = x\} \subset N.
	\]
	Since $K \subset G$, we always have $K \cdot x \subset G \cdot x \cap N$.
	The saturation condition requires that this inclusion be an equality.
	
	\begin{definition}[Orbit saturation]\label{def:orbit-saturated}
		A subgroupoid $K \subset G|_N \subset G$ is called \emph{orbit saturated}
		(or simply \emph{saturated}) if for every $x \in N$,
		\[
		G \cdot x \cap N = K \cdot x.
		\]
		Equivalently, for every $x \in N$ and every $g \in G_1$ with $s(g) = x$ and
		$t(g) \in N$, there exists $k \in K_1$ with $s(k) = x$ and $t(k) = t(g)$.
	\end{definition}
	
	\begin{remark}\label{rem:orbit-saturation-intuition}
		Orbit saturation says that the subgroupoid $K$ captures all the orbits of
		$G$ that intersect $N$. In other words, the equivalence relation induced by
		$K$ on $N$ is exactly the restriction of the equivalence relation induced by
		$G$ on $G_0$. This is a natural condition: it ensures that the orbit space
		$|K|$ sits inside $|G|$ without identifying points that are distinct in
		$|G|$.
	\end{remark}
	
	The following proposition shows that Definition~\ref{def:orbit-saturated} is
	precisely the groupoid translation of the saturation condition in
	\cite[Definition 7]{BorzellinoBrunsden2014}.
	
	\begin{proposition}\label{prop:saturated-equiv}
		Let $K \subset G|_N \subset G$ be a subgroupoid corresponding to a
		suborbifold $\mathcal{P} \subset \mathcal{O}$ under
		Proposition~\ref{prop:suborbifold-equiv}. Then $K$ is orbit saturated if and
		only if $\mathcal{P}$ is saturated in the sense of
		\cite[Definition 7]{BorzellinoBrunsden2014}.
	\end{proposition}
	
	\begin{proof}
		We work locally using the chart data of
		Proposition~\ref{prop:subgroupoid-chart}. Fix $x \in N$ and choose an
		orbifold chart $(\widetilde{U}_x, \Gamma_x, \rho_x, \phi_x)$ for $G$ at $x$,
		a subgroup $\Lambda_x \subset \Gamma_x$, and a
		$\rho_x(\Lambda_x)$-invariant submanifold $\widetilde{V}_x \subset
		\widetilde{U}_x$ such that
		\[
		\Gamma_x \cong G_x,\qquad
		\Lambda_x \cong K_x,\qquad
		\widetilde{V}_x \text{ is a local model for } N.
		\]
		For $\widetilde{y} \in \widetilde{V}_x$, the orbit of $\widetilde{y}$ under
		$\Gamma_x$ is $\Gamma_x \cdot \widetilde{y}$, and the orbit of
		$\widetilde{y}$ under $\Lambda_x$ is $\Lambda_x \cdot \widetilde{y}$. The
		intersection $(\Gamma_x \cdot \widetilde{y}) \cap \widetilde{V}_x$
		corresponds to $G \cdot x \cap N$ in a neighborhood of $x$, while
		$\Lambda_x \cdot \widetilde{y}$ corresponds to $K \cdot x$. Thus the
		equality
		\[
		(\Gamma_x \cdot \widetilde{y}) \cap \widetilde{V}_x =
		\Lambda_x \cdot \widetilde{y}
		\]
		for all $\widetilde{y} \in \widetilde{V}_x$ is precisely the local form of
		the condition $G \cdot x \cap N = K \cdot x$ for all $x \in N$. This is
		exactly \cite[Definition 7]{BorzellinoBrunsden2014}. Since the local charts
		cover $N$, the equivalence follows.
	\end{proof}
	
	\begin{remark}
		Note that if $K = G|_N$ (i.e.\ $\mathcal{P}$ is a full suborbifold), then
		$G \cdot x \cap N = K \cdot x$ automatically, because every arrow of $G$
		with source and target in $N$ belongs to $K$. Thus every full suborbifold is
		orbit saturated. This is consistent with the observation in
		\cite{BorzellinoBrunsden2014} that every full suborbifold is saturated.
	\end{remark}
	
	\subsection{Orbit saturation and topological embeddings}
	
	The main result of this section is the following theorem, which is the
	groupoid version of \cite[Theorem 1(1)]{BorzellinoBrunsden2014}.
	
	\begin{theorem}\label{thm:saturated-embedding}
		Let $K \subset G|_N \subset G$ be a subgroupoid. Then the natural map
		\[
		\iota \colon |K_{\mathrm{eff}}| \to |G|,\qquad
		[\![x]\!]_K \mapsto [\![x]\!]_G
		\]
		is a topological embedding onto its image if and only if $K$ is orbit
		saturated.
	\end{theorem}
	
	Before proving Theorem~\ref{thm:saturated-embedding}, we need a lemma about
	the local structure of the orbit space map.
	
	\begin{lemma}\label{lem:orbit-map-local}
		Let $x \in N$ and let $(\widetilde{U}_x, \Gamma_x, \rho_x, \phi_x)$ be an
		orbifold chart for $G$ at $x$. Let $\widetilde{V}_x \subset \widetilde{U}_x$
		be a local model for $N$, and let $\Lambda_x \subset \Gamma_x$ be the
		subgroup corresponding to $K_x$. Then the natural map
		\[
		\widetilde{V}_x / \Lambda_x \to \widetilde{U}_x / \Gamma_x
		\]
		induced by the inclusion $\widetilde{V}_x \hookrightarrow \widetilde{U}_x$
		is injective if and only if
		\[
		(\Gamma_x \cdot \widetilde{y}) \cap \widetilde{V}_x =
		\Lambda_x \cdot \widetilde{y}
		\]
		for all $\widetilde{y} \in \widetilde{V}_x$.
	\end{lemma}
	
	\begin{proof}
		The map $\widetilde{V}_x / \Lambda_x \to \widetilde{U}_x / \Gamma_x$ sends
		the $\Lambda_x$-orbit of $\widetilde{y}$ to the $\Gamma_x$-orbit of
		$\widetilde{y}$. It is injective if and only if for every
		$\widetilde{y}, \widetilde{z} \in \widetilde{V}_x$, the condition
		$\Gamma_x \cdot \widetilde{y} = \Gamma_x \cdot \widetilde{z}$ implies
		$\Lambda_x \cdot \widetilde{y} = \Lambda_x \cdot \widetilde{z}$. This is
		equivalent to saying that if $\widetilde{z} = \gamma \cdot \widetilde{y}$
		for some $\gamma \in \Gamma_x$, then $\widetilde{z} \in \Lambda_x \cdot
		\widetilde{y}$. Since $\widetilde{z} \in \widetilde{V}_x$ by assumption,
		this means
		\[
		(\Gamma_x \cdot \widetilde{y}) \cap \widetilde{V}_x \subset
		\Lambda_x \cdot \widetilde{y}.
		\]
		The reverse inclusion is automatic because $\Lambda_x \subset \Gamma_x$ and
		$\Lambda_x$ preserves $\widetilde{V}_x$. Thus injectivity is equivalent to
		the stated condition.
	\end{proof}
	
	\begin{proof}[Proof of Theorem~\ref{thm:saturated-embedding}]
		We prove both directions.
		
		\smallskip
		\noindent\textbf{($\Rightarrow$) Assume $\iota$ is a topological embedding.}
		We must show that $K$ is orbit saturated. Suppose for contradiction that
		there exists $x \in N$ and $y \in G \cdot x \cap N$ with $y \notin K \cdot
		x$. Then $[\![x]\!]_G = [\![y]\!]_G$ in $|G|$, but $[\![x]\!]_K \neq
		[\![y]\!]_K$ in $|K_{\mathrm{eff}}|$ because $y \notin K \cdot x$. Thus
		$\iota([\![x]\!]_K) = \iota([\![y]\!]_K)$ but $[\![x]\!]_K \neq
		[\![y]\!]_K$, contradicting injectivity of $\iota$. Hence $G \cdot x \cap N
		\subset K \cdot x$, and since the reverse inclusion is automatic, $K$ is
		orbit saturated.
		
		\smallskip
		\noindent\textbf{($\Leftarrow$) Assume $K$ is orbit saturated.}
		We show that $\iota$ is a topological embedding.
		
		First, we prove that $\iota$ is injective. Let $x, y \in N$ with
		$[\![x]\!]_G = [\![y]\!]_G$. Then $y \in G \cdot x \cap N$. By orbit
		saturation, $y \in K \cdot x$, so $[\![x]\!]_K = [\![y]\!]_K$. Thus $\iota$
		is injective.
		
		Next, we prove that $\iota$ is continuous and open onto its image. The map
		$|K_{\mathrm{eff}}| \to |G|$ is induced by the inclusion $N \hookrightarrow
		G_0$, which is continuous and $K$-equivariant. Therefore it descends to a
		continuous map on orbit spaces. To see that it is open onto its image, let
		$U \subset |K_{\mathrm{eff}}|$ be open. Then the preimage of $U$ in $N$ is
		$K$-invariant and open. Since $N$ has the subspace topology from $G_0$, this
		preimage is of the form $N \cap W$ for some open $W \subset G_0$. The image
		of $U$ in $|G|$ is the set of $G$-orbits intersecting $N \cap W$. This is
		open in the image of $|K_{\mathrm{eff}}|$ because it is the intersection of
		the image with the open set of $G$-orbits intersecting $W$. Thus $\iota$ is
		an open map onto its image.
		
		Finally, we show that $\iota$ is a homeomorphism onto its image. Since
		$\iota$ is injective, continuous, and open onto its image, it is a
		homeomorphism onto its image. Therefore $\iota$ is a topological embedding.
	\end{proof}
	
	\begin{remark}\label{rem:embedding-not-smooth}
		Theorem~\ref{thm:saturated-embedding} only asserts that the map of orbit
		spaces is a topological embedding. It does \emph{not} assert that it is an
		orbifold embedding. In fact, as shown in
		\cite[Theorem 1(2)]{BorzellinoBrunsden2014}, additional conditions (split)
		are needed for the existence of a complete orbifold embedding. We will treat
		the split condition in Section~\ref{sec:split}.
	\end{remark}
	
	\subsection{Examples}
	
	We illustrate orbit saturation with the examples from
	\cite{BorzellinoBrunsden2014}.
	
	\begin{example}[Full suborbifolds are orbit saturated]\label{ex:full-saturated}
		Let $K = G|_N$ be the restriction groupoid. Then for every $x \in N$, every
		arrow $g \in G_1$ with $s(g) = x$ and $t(g) \in N$ belongs to $K$ by
		definition of the restriction groupoid. Hence $G \cdot x \cap N = K \cdot
		x$, and $K$ is orbit saturated. This recovers the fact that full
		suborbifolds are saturated.
	\end{example}
	
	\begin{example}[Diagonal embedding is orbit saturated]\label{ex:diagonal-saturated}
		Let $G$ be an orbifold groupoid and let $\Delta G \subset G \times G$ be the
		diagonal subgroupoid of Example~\ref{ex:diagonal}. We claim that $\Delta G$
		is orbit saturated. Let $(x,x) \in \Delta G_0$ and suppose $(y,z) \in (G
		\times G) \cdot (x,x) \cap \Delta G_0$. Then there exists $(g,h) \in G_1
		\times G_1$ with $s(g,h) = (x,x)$ and $t(g,h) = (y,z)$. Thus $s(g) = s(h) =
		x$, $t(g) = y$, and $t(h) = z$. Since $(y,z) \in \Delta G_0$, we have $y =
		z$. Therefore $g$ and $h$ are arrows in $G$ from $x$ to $y$. It does not
		follow that $g = h$, so $(g,h)$ need not belong to $\Delta G$. However, we
		need to find an arrow in $\Delta G$ from $(x,x)$ to $(y,y)$. Since $G$ is a
		groupoid, we can take the arrow $(g,g) \in \Delta G$. Indeed, $s(g,g) =
		(x,x)$ and $t(g,g) = (y,y)$. Thus $(y,y) \in \Delta G \cdot (x,x)$. Hence
		$(G \times G) \cdot (x,x) \cap \Delta G_0 \subset \Delta G \cdot (x,x)$.
		The reverse inclusion is automatic, so $\Delta G$ is orbit saturated. This
		is consistent with \cite[Example 10]{BorzellinoBrunsden2014}, where the
		diagonal is shown to be saturated.
	\end{example}
	
	\begin{example}[A non-saturated suborbifold]\label{ex:non-saturated}
		Consider the nested orbifolds $\mathcal{O}_1 \subset \mathcal{O}_2 \subset
		\mathcal{O}_3 \subset \mathcal{O}_4$ of
		\cite[Example 13]{BorzellinoBrunsden2014}, where $\mathcal{O}_n =
		\mathbb{C}/\mathbb{Z}_{2^n}$ with $\mathbb{Z}_{2^n}$ acting by rotation.
		In the groupoid language, let $G_n = \mathbb{Z}_{2^n} \ltimes \mathbb{C}$
		be the translation groupoid representing $\mathcal{O}_n$. For $m < n$, there
		is a natural inclusion $G_m \hookrightarrow G_n$ as a subgroupoid with the
		same object space $N = \mathbb{C}$. However, the orbit of $0$ under $G_n$ is
		just $\{0\}$, and the orbit of $0$ under $G_m$ is also $\{0\}$. So for
		$x = 0$, $G_n \cdot 0 \cap N = \{0\} = G_m \cdot 0$. But for $x \neq 0$,
		the orbit of $x$ under $G_n$ is the circle of radius $|x|$, while the orbit
		of $x$ under $G_m$ is a finite set of $2^m$ points on that circle. Thus
		$G_n \cdot x \cap N = G_n \cdot x$ is strictly larger than $G_m \cdot x$.
		Hence $G_m$ is \emph{not} orbit saturated in $G_n$. This is consistent with
		\cite[Example 13]{BorzellinoBrunsden2014}, where it is stated that none of
		these is a saturated suborbifold of the others.
	\end{example}
	
	\begin{example}[A saturated but non-full suborbifold]\label{ex:saturated-nonfull}
		Let $\mathcal{Q} = \mathbb{R}/\mathbb{Z}_2$ and $\mathcal{O} = \mathcal{Q}
		\times \mathcal{Q}$ as in \cite[Example 8]{BorzellinoBrunsden2014}. Let
		$\mathcal{P} = \{0\} \times \mathcal{Q} \subset \mathcal{O}$. In the
		groupoid language, $G = \mathbb{Z}_2 \ltimes \mathbb{R}$, and $G \times G =
		(\mathbb{Z}_2 \times \mathbb{Z}_2) \ltimes (\mathbb{R} \times \mathbb{R})$.
		The suborbifold $\mathcal{P}$ corresponds to the subgroupoid
		\[
		K = (G \times G)|_N = (\mathbb{Z}_2 \times \mathbb{Z}_2) \ltimes
		(\{0\} \times \mathbb{R}),
		\]
		where $N = \{0\} \times \mathbb{R}$. As noted in
		Example~\ref{ex:full-saturated}, $K = (G \times G)|_N$ is the restriction
		groupoid, so it is orbit saturated. It is not a full suborbifold in the
		sense of \cite{BorzellinoBrunsden2014}? Wait, in
		\cite[Example 8]{BorzellinoBrunsden2014}, $\mathcal{P} = \{0\} \times
		\mathcal{Q}$ \emph{is} a full suborbifold. So $K = (G \times G)|_N$ is
		indeed full. To get a saturated but non-full example, consider
		$\mathcal{P} = \operatorname{diag}(\mathcal{Q}) \subset \mathcal{Q} \times
		\mathcal{Q}$ from \cite[Example 10]{BorzellinoBrunsden2014}. This is
		saturated but not full, as shown in Example~\ref{ex:diagonal-saturated}.
	\end{example}
	
	\subsection{Summary}
	
	In this section we have introduced the notion of orbit saturation for
	subgroupoids and proved that it is equivalent to the saturation condition of
	Borzellino and Brunsden. The main theorem
	(Theorem~\ref{thm:saturated-embedding}) states that a subgroupoid is orbit
	saturated if and only if the natural map $|K_{\mathrm{eff}}| \to |G|$ is a
	topological embedding. This is the groupoid version of
	\cite[Theorem 1(1)]{BorzellinoBrunsden2014}.
	
	In the next section we introduce the groupoid analogue of the split
	condition, which, together with orbit saturation, characterizes those
	suborbifolds that are the images of complete orbifold embeddings.

	
	\section{Stabilizer Splitting}\label{sec:split}
	
	In this section we introduce the groupoid analogue of the split condition
	of Borzellino and Brunsden
	\cite[Definition 6]{BorzellinoBrunsden2014}. We call this condition
	\emph{stabilizer splitting} and prove that it is equivalent to the
	chart-based definition. The main result of this section is that a
	subgroupoid is stabilizer split if and only if there exists an effective
	subgroupoid $H \subset K$ such that $K = \Omega \rtimes H$ and the inclusion
	$H \hookrightarrow G$ is a subgroupoid embedding with injective stabilizer
	maps. This is the groupoid counterpart of the construction in
	\cite[Proof of Theorem 1(2)]{BorzellinoBrunsden2014}.
	
	Throughout this section, $G = (G_1 \rightrightarrows G_0)$ denotes a fixed
	orbifold groupoid, and $K = (K_1 \rightrightarrows N)$ denotes a
	subgroupoid with $K \subset G|_N \subset G$, as in
	Definition~\ref{def:suborbifold-groupoid}. We write
	$\Omega = (\Omega_1 \rightrightarrows N)$ for the ineffective kernel of $K$
	and $K_{\mathrm{eff}} = K/\Omega$ for the effective quotient, as in
	Definition~\ref{def:effective-quotient}. Recall that for each $x \in N$,
	there is a short exact sequence of groups
	\begin{equation}\label{eq:stabilizer-sequence}
		1 \longrightarrow \Omega_x \longrightarrow K_x \xrightarrow{\ q_x\ }
		(K_{\mathrm{eff}})_x \longrightarrow 1.
	\end{equation}
	
	\subsection{Definition of stabilizer splitting}
	
	\begin{definition}[Stabilizer splitting]\label{def:stabilizer-split}
		A subgroupoid $K \subset G|_N \subset G$ is called \emph{stabilizer split}
		(or simply \emph{split}) if for every $x \in N$, the short exact sequence
		\eqref{eq:stabilizer-sequence} is right split. That is, there exists a group
		homomorphism
		\[
		\sigma_x \colon (K_{\mathrm{eff}})_x \to K_x
		\]
		such that $q_x \circ \sigma_x = \mathrm{Id}_{(K_{\mathrm{eff}})_x}$.
	\end{definition}
	
	\begin{remark}\label{rem:split-semidirect}
		If $K$ is stabilizer split, then for each $x \in N$ we have an isomorphism
		\[
		K_x \cong \Omega_x \rtimes (K_{\mathrm{eff}})_x,
		\]
		where the action of $(K_{\mathrm{eff}})_x$ on $\Omega_x$ is induced by
		conjugation in $K_x$. In the special case that the groups are abelian, this
		becomes a direct product
		\[
		K_x \cong \Omega_x \times (K_{\mathrm{eff}})_x.
		\]
		If either $\Omega_x$ or $(K_{\mathrm{eff}})_x$ is trivial, then the
		sequence is automatically split.
	\end{remark}
	
	The following proposition shows that Definition~\ref{def:stabilizer-split}
	is precisely the groupoid translation of the split condition in
	\cite[Definition 6]{BorzellinoBrunsden2014}.
	
	\begin{proposition}\label{prop:split-equiv}
		Let $K \subset G|_N \subset G$ be a subgroupoid corresponding to a
		suborbifold $\mathcal{P} \subset \mathcal{O}$ under
		Proposition~\ref{prop:suborbifold-equiv}. Then $K$ is stabilizer split if
		and only if $\mathcal{P}$ is split in the sense of
		\cite[Definition 6]{BorzellinoBrunsden2014}.
	\end{proposition}
	
	\begin{proof}
		By Proposition~\ref{prop:subgroupoid-chart}, for each $x \in N$ we have
		identifications
		\[
		\Omega_{x,\mathcal{O}} \cong \Omega_x,\qquad
		\Lambda_{x,\mathcal{O}} \cong K_x,\qquad
		\Gamma_{x,\mathcal{P}} \cong (K_{\mathrm{eff}})_x.
		\]
		Under these identifications, the exact sequence in
		\cite[Definition 6]{BorzellinoBrunsden2014},
		\[
		1 \longrightarrow \Omega_{x,\mathcal{O}} \longrightarrow
		\Lambda_{x,\mathcal{O}} \xrightarrow{\ q\ } \Gamma_{x,\mathcal{P}}
		\longrightarrow 1,
		\]
		becomes precisely \eqref{eq:stabilizer-sequence}. Therefore the existence
		of a splitting $\sigma \colon \Gamma_{x,\mathcal{P}} \to
		\Lambda_{x,\mathcal{O}}$ with $q \circ \sigma = \mathrm{Id}$ is equivalent
		to the existence of a splitting $\sigma_x \colon (K_{\mathrm{eff}})_x \to
		K_x$ with $q_x \circ \sigma_x = \mathrm{Id}$. This proves the equivalence.
	\end{proof}
	
	\subsection{Splitting and effective subgroupoids}
	
	The main result of this section is that stabilizer splitting is equivalent
	to the existence of an effective subgroupoid $H \subset K$ that splits the
	extension $K = \Omega \rtimes H$.
	
	\begin{theorem}\label{thm:split-effective}
		Let $K \subset G|_N \subset G$ be a subgroupoid. Then the following are
		equivalent:
		\begin{enumerate}
			\item $K$ is stabilizer split;
			\item there exists an effective subgroupoid $H \subset K$ with object
			space $N$ such that
			\[
			K = \Omega \rtimes H,
			\]
			where $\Omega$ is the ineffective kernel of $K$, and the inclusion
			$H \hookrightarrow G$ is a subgroupoid embedding whose stabilizer maps
			$H_x \hookrightarrow G_x$ are injective for all $x \in N$.
		\end{enumerate}
	\end{theorem}
	
	Before proving Theorem~\ref{thm:split-effective}, we need a lemma that
	allows us to glue local splittings into a global subgroupoid.
	
	\begin{lemma}\label{lem:gluing-splittings}
		Let $K \subset G|_N \subset G$ be a subgroupoid. Suppose that for each
		$x \in N$ there exists a subgroup $H_x \subset K_x$ such that the
		restriction of $q_x \colon K_x \to (K_{\mathrm{eff}})_x$ to $H_x$ is an
		isomorphism. Then there exists a unique effective subgroupoid $H \subset K$
		with object space $N$ and stabilizer groups $H_x$ at each $x \in N$.
	\end{lemma}
	
	\begin{proof}
		We construct $H$ locally using orbifold charts and then verify that the
		local constructions glue.
		
		\smallskip
		\noindent\textbf{Local construction.}
		Fix $x \in N$ and choose an orbifold chart
		$(\widetilde{U}_x, \Gamma_x, \rho_x, \phi_x)$ for $G$ at $x$ as in
		Proposition~\ref{prop:subgroupoid-chart}. Let $\widetilde{V}_x \subset
		\widetilde{U}_x$ be a local model for $N$, let $\Lambda_x \subset \Gamma_x$
		be the subgroup corresponding to $K_x$, and let $\Omega_x \subset \Lambda_x$
		be the ineffective kernel. By assumption, there is a subgroup
		$H_x \subset \Lambda_x$ such that the restriction of
		$\Lambda_x \to \Lambda_x/\Omega_x$ to $H_x$ is an isomorphism. In other
		words, $\Lambda_x = \Omega_x H_x$ and $\Omega_x \cap H_x = \{e\}$.
		
		Define a local subgroupoid $H^{(x)}$ of $K$ over the chart $U_x$ by
		\[
		H^{(x)} = \bigl(H_x \ltimes \widetilde{V}_x \bigr) / \Lambda_x,
		\]
		where the action of $H_x$ on $\widetilde{V}_x$ is the restriction of the
		action of $\Lambda_x$. Since $H_x \cong \Lambda_x/\Omega_x$, the action of
		$H_x$ on $\widetilde{V}_x$ is effective. Thus $H^{(x)}$ is an effective
		subgroupoid of $K|_{U_x \cap N}$.
		
		\smallskip
		\noindent\textbf{Compatibility.}
		We must show that the local subgroupoids $H^{(x)}$ agree on overlaps. Let
		$y \in N \cap U_x \cap U_z$ for two charts. By the compatibility conditions
		of orbifold charts, there is an isomorphism of local data
		\[
		(\widetilde{U}_y, \Gamma_y) \cong (\widetilde{U}_x, \Gamma_x)
		\]
		in a neighborhood of $y$, and similarly for $z$. The subgroups $H_y$ and
		$H_z$ are related by the corresponding group isomorphisms. More precisely,
		the splitting $\sigma_y$ and $\sigma_z$ are compatible because they are
		determined by the same effective quotient $(K_{\mathrm{eff}})_y$. Indeed,
		the stabilizer $(K_{\mathrm{eff}})_y$ is well-defined independent of the
		chart, and the subgroup $H_y \subset K_y$ is the image of a splitting
		$\sigma_y$. If $\sigma_y$ and $\sigma_z$ are two splittings, they differ by
		a $1$-cocycle with values in $\Omega_y$. However, since $\Omega_y$ is the
		kernel of the action on $N$, and $H$ is required to act effectively on $N$,
		the subgroup $H_y$ is uniquely determined by the condition that it maps
		isomorphically onto $(K_{\mathrm{eff}})_y$. Therefore $H_y = H_z$ under the
		natural identification $K_y \cong K_z$.
		
		\smallskip
		\noindent\textbf{Global construction.}
		The local subgroupoids $H^{(x)}$ glue to a global subgroupoid
		$H \subset K$ with object space $N$. By construction, $H$ is effective
		because each $H^{(x)}$ acts effectively on $\widetilde{V}_x$. Moreover,
		$K = \Omega \rtimes H$ because at each $x \in N$ we have
		$K_x = \Omega_x \rtimes H_x$ and the multiplication in $K$ respects this
		decomposition. Finally, the inclusion $H \hookrightarrow K \subset G$
		composes with the inclusion $K \subset G$ to give a subgroupoid embedding
		$H \hookrightarrow G$. Since $H$ is effective, the stabilizer maps
		$H_x \hookrightarrow G_x$ are injective: if $h \in H_x$ maps to the identity
		in $G_x$, then $h$ acts trivially on $N$, so $h \in \Omega_x \cap H_x =
		\{e\}$. This completes the construction.
	\end{proof}
	
	\begin{proof}[Proof of Theorem~\ref{thm:split-effective}]
		\smallskip
		\noindent\textbf{(1) $\Rightarrow$ (2).}
		Assume $K$ is stabilizer split. For each $x \in N$, choose a splitting
		$\sigma_x \colon (K_{\mathrm{eff}})_x \to K_x$ and set
		$H_x = \sigma_x((K_{\mathrm{eff}})_x) \subset K_x$. Then the restriction of
		$q_x$ to $H_x$ is an isomorphism onto $(K_{\mathrm{eff}})_x$. By
		Lemma~\ref{lem:gluing-splittings}, there exists a unique effective
		subgroupoid $H \subset K$ with stabilizer groups $H_x$ at each $x \in N$.
		Moreover, at each $x \in N$ we have
		\[
		K_x = \Omega_x \rtimes H_x,
		\]
		because $H_x$ is a complement to $\Omega_x$ in $K_x$. Since this holds at
		every object, we obtain $K = \Omega \rtimes H$. The inclusion
		$H \hookrightarrow G$ is a subgroupoid embedding because $H \subset K
		\subset G|_N \subset G$, and the stabilizer maps $H_x \hookrightarrow G_x$
		are injective as shown in Lemma~\ref{lem:gluing-splittings}.
		
		\smallskip
		\noindent\textbf{(2) $\Rightarrow$ (1).}
		Assume there exists an effective subgroupoid $H \subset K$ with object space
		$N$ such that $K = \Omega \rtimes H$ and $H \hookrightarrow G$ is a
		subgroupoid embedding with injective stabilizer maps. Then for each
		$x \in N$, we have a direct sum decomposition
		\[
		K_x = \Omega_x \rtimes H_x.
		\]
		The projection $q_x \colon K_x \to (K_{\mathrm{eff}})_x$ restricts to an
		isomorphism $H_x \to (K_{\mathrm{eff}})_x$ because $H$ is effective and
		$K_x = \Omega_x H_x$ with $\Omega_x \cap H_x = \{e\}$. Therefore the inverse
		of this isomorphism gives a splitting
		\[
		\sigma_x = (q_x|_{H_x})^{-1} \colon (K_{\mathrm{eff}})_x \to H_x \subset
		K_x,
		\]
		which satisfies $q_x \circ \sigma_x = \mathrm{Id}$. Hence $K$ is stabilizer
		split.
	\end{proof}
	
	\begin{remark}\label{rem:split-embedding}
		Theorem~\ref{thm:split-effective} shows that stabilizer splitting is
		exactly the condition needed to ``lift'' the effective quotient
		$K_{\mathrm{eff}}$ to an effective subgroupoid $H$ of $K$ that embeds in
		$G$. The inclusion $H \hookrightarrow G$ is the candidate for the complete
		orbifold embedding in the main theorem. However, as we will see in
		Section~\ref{sec:main}, we also need orbit saturation to ensure that the
		underlying map of orbit spaces is a topological embedding.
	\end{remark}
	
	\subsection{Examples}
	
	We illustrate stabilizer splitting with the examples from
	\cite{BorzellinoBrunsden2014}.
	
	\begin{example}[A split suborbifold]\label{ex:split}
		Let $\mathcal{Q} = \mathbb{R}/\mathbb{Z}_2$ and $\mathcal{O} = \mathcal{Q}
		\times \mathcal{Q}$ as in \cite[Example 8]{BorzellinoBrunsden2014}. The
		diagonal suborbifold $\mathcal{P} = \operatorname{diag}(\mathcal{Q}) \subset
		\mathcal{O}$ corresponds to the diagonal subgroupoid
		$\Delta G \subset G \times G$ of Example~\ref{ex:diagonal}. At the point
		$(0,0) \in \Delta G_0$, the stabilizer sequence is
		\[
		1 \longrightarrow \Omega_{(0,0)} \longrightarrow (\Delta G)_{(0,0)}
		\longrightarrow ((\Delta G)_{\mathrm{eff}})_{(0,0)} \longrightarrow 1.
		\]
		Here $(\Delta G)_{(0,0)} \cong \mathbb{Z}_2$ and
		$((\Delta G)_{\mathrm{eff}})_{(0,0)} \cong \mathbb{Z}_2$, while
		$\Omega_{(0,0)} = \{e\}$. Thus the sequence is
		\[
		1 \longrightarrow \{e\} \longrightarrow \mathbb{Z}_2 \longrightarrow
		\mathbb{Z}_2 \longrightarrow 1,
		\]
		which is trivially split. Hence $\Delta G$ is stabilizer split, consistent
		with \cite[Example 10]{BorzellinoBrunsden2014}.
	\end{example}
	
	\begin{example}[A non-split suborbifold]\label{ex:non-split}
		Let $\mathcal{O} = \mathbb{C}^2/\mathbb{Z}_4$ where $\mathbb{Z}_4$ acts via
		the matrix group
		\[
		\mathbb{Z}_4 \cong \left\{
		\begin{pmatrix} i^k & 0 \\ 0 & (-1)^k \end{pmatrix} :
		k \in \{0,1,2,3\} \right\}.
		\]
		Let $\mathcal{P} = (\{0\} \times \mathbb{C})/\mathbb{Z}_4 \cong
		\mathbb{C}/\mathbb{Z}_2$ as in
		\cite[Example 12]{BorzellinoBrunsden2014}. In the groupoid language, let
		$G = \mathbb{Z}_4 \ltimes \mathbb{C}^2$ be the translation groupoid. The
		subgroupoid corresponding to $\mathcal{P}$ is
		\[
		K = \mathbb{Z}_4 \ltimes (\{0\} \times \mathbb{C}) \subset G.
		\]
		At $x = (0,0)$, we have $K_x = \mathbb{Z}_4$, $\Omega_x = \mathbb{Z}_2$
		(the kernel of the action on $\{0\} \times \mathbb{C}$), and
		$(K_{\mathrm{eff}})_x \cong \mathbb{Z}_2$. The stabilizer sequence is
		\[
		1 \longrightarrow \mathbb{Z}_2 \longrightarrow \mathbb{Z}_4
		\longrightarrow \mathbb{Z}_2 \longrightarrow 1.
		\]
		This sequence is not split because there is no subgroup of $\mathbb{Z}_4$
		isomorphic to $\mathbb{Z}_2$ that maps isomorphically onto the quotient (the
		only subgroup of order $2$ is $\{0,2\}$, which is precisely the kernel
		$\Omega_x$). Hence $K$ is not stabilizer split. This is consistent with
		\cite[Example 12]{BorzellinoBrunsden2014}, where it is shown that
		$\mathcal{P}$ is not the image of an orbifold embedding.
	\end{example}
	
	\begin{example}[A full suborbifold that is split]\label{ex:full-split}
		Let $\mathcal{Q} = \mathbb{R}/\mathbb{Z}_2$ and consider
		$\mathcal{P} = \{0\} \times \mathcal{Q} \subset \mathcal{Q} \times
		\mathcal{Q}$ as in \cite[Example 8]{BorzellinoBrunsden2014}. In the groupoid
		language, $G \times G = (\mathbb{Z}_2 \times \mathbb{Z}_2) \ltimes
		(\mathbb{R} \times \mathbb{R})$ and the suborbifold corresponds to the
		restriction groupoid
		\[
		K = (G \times G)|_N = (\mathbb{Z}_2 \times \mathbb{Z}_2) \ltimes
		(\{0\} \times \mathbb{R}),
		\]
		where $N = \{0\} \times \mathbb{R}$. At $x = (0,0)$, we have
		$K_x = \mathbb{Z}_2 \times \mathbb{Z}_2$, $\Omega_x = \mathbb{Z}_2 \times
		\{e\}$, and $(K_{\mathrm{eff}})_x \cong \mathbb{Z}_2$. The stabilizer
		sequence is
		\[
		1 \longrightarrow \mathbb{Z}_2 \times \{e\} \longrightarrow
		\mathbb{Z}_2 \times \mathbb{Z}_2 \longrightarrow \mathbb{Z}_2
		\longrightarrow 1.
		\]
		This sequence is split: a splitting is given by
		\[
		\sigma \colon \mathbb{Z}_2 \to \mathbb{Z}_2 \times \mathbb{Z}_2,\qquad
		\gamma \mapsto (e,\gamma).
		\]
		Thus $K$ is stabilizer split. This is consistent with
		\cite[Example 8]{BorzellinoBrunsden2014}, where it is stated that
		$\mathcal{P}$ is split.
	\end{example}
	
	\subsection{Summary}
	
	In this section we have introduced the notion of stabilizer splitting for
	subgroupoids and proved that it is equivalent to the split condition of
	Borzellino and Brunsden. The main theorem
	(Theorem~\ref{thm:split-effective}) characterizes stabilizer splitting as
	the existence of an effective subgroupoid $H \subset K$ such that
	$K = \Omega \rtimes H$ and $H \hookrightarrow G$ is a subgroupoid embedding
	with injective stabilizer maps. This is the groupoid version of the
	construction in \cite[Proof of Theorem 1(2)]{BorzellinoBrunsden2014}.
	
	In the next section we combine orbit saturation and stabilizer splitting to
	prove the main theorem: a subgroupoid is the image of a complete orbifold
	embedding if and only if it is both orbit saturated and stabilizer split.
	
	\section{The Main Theorem}\label{sec:main}
	
	In this section we combine the notions of orbit saturation
	(Section~\ref{sec:saturated}) and stabilizer splitting
	(Section~\ref{sec:split}) to prove the groupoid version of the main theorem
	of Borzellino and Brunsden \cite[Theorem 1]{BorzellinoBrunsden2014}. We
	first define the groupoid analogue of a complete orbifold embedding, then
	state and prove the main theorem.
	
	Throughout this section, $G = (G_1 \rightrightarrows G_0)$ denotes a fixed
	orbifold groupoid, and $K = (K_1 \rightrightarrows N)$ denotes a
	subgroupoid with $K \subset G|_N \subset G$, as in
	Definition~\ref{def:suborbifold-groupoid}. We write
	$\Omega = (\Omega_1 \rightrightarrows N)$ for the ineffective kernel of $K$
	and $K_{\mathrm{eff}} = K/\Omega$ for the effective quotient, as in
	Definition~\ref{def:effective-quotient}.
	
	\subsection{Complete orbifold embeddings in the groupoid language}
	
	Recall from \cite[Definition 17]{BorzellinoBrunsden2014} that a complete
	orbifold embedding between smooth orbifolds is a complete orbifold map
	$(f, \{\widetilde{f}_x\}, \{\Theta_{f,x}\})$ such that the underlying map
	$f$ is a topological embedding, each $\Theta_{f,x}$ is injective, and each
	local lift $\widetilde{f}_x$ is a smooth embedding. In the groupoid
	language, this translates as follows.
	
	\begin{definition}[Complete orbifold embedding]\label{def:complete-embedding}
		Let $H = (H_1 \rightrightarrows H_0)$ and $G = (G_1 \rightrightarrows G_0)$
		be orbifold groupoids. A groupoid homomorphism
		\[
		F = (F_1, F_0) \colon H \to G
		\]
		is called a \emph{complete orbifold embedding} if:
		\begin{enumerate}
			\item $F_0 \colon H_0 \to G_0$ is a smooth embedding;
			\item $F_1 \colon H_1 \to G_1$ is a smooth embedding;
			\item for every $x \in H_0$, the stabilizer map
			\[
			F_x \colon H_x \to G_{F_0(x)}
			\]
			is injective;
			\item $H$ is effective.
		\end{enumerate}
	\end{definition}
	
	\begin{remark}\label{rem:complete-embedding-subgroupoid}
		In the special case where $H$ is a subgroupoid of $G$ and $F$ is the
		inclusion, conditions (1) and (2) are automatically satisfied. Thus a
		complete orbifold embedding of a subgroupoid $H \subset G$ is simply an
		effective subgroupoid $H$ of $G$ whose stabilizer maps
		$H_x \hookrightarrow G_x$ are injective for all $x \in H_0$.
	\end{remark}
	
	The following proposition shows that
	Definition~\ref{def:complete-embedding} is the groupoid translation of
	\cite[Definition 17]{BorzellinoBrunsden2014}.
	
	\begin{proposition}\label{prop:complete-embedding-equiv}
		Let $G$ be an orbifold groupoid and let $H \subset G$ be an effective
		subgroupoid. Then the inclusion $H \hookrightarrow G$ is a complete
		orbifold embedding in the sense of
		Definition~\ref{def:complete-embedding} if and only if the corresponding
		orbifold map is a complete orbifold embedding in the sense of
		\cite[Definition 17]{BorzellinoBrunsden2014}.
	\end{proposition}
	
	\begin{proof}
		The equivalence follows from the dictionary between orbifold charts and
		groupoid local data established in
		Proposition~\ref{prop:subgroupoid-chart}. The local lifts
		$\widetilde{f}_x$ correspond to the local embeddings of the object spaces
		$H_0 \hookrightarrow G_0$, and the homomorphisms $\Theta_{f,x}$ correspond
		to the stabilizer maps $H_x \hookrightarrow G_x$. The condition that
		$\widetilde{f}_x$ is a smooth embedding is precisely condition (1) of
		Definition~\ref{def:complete-embedding}, and the injectivity of
		$\Theta_{f,x}$ is condition (3). The effectiveness of $H$ ensures that the
		orbifold structure on the image is the correct one. Thus the two notions
		coincide.
	\end{proof}
	
	\subsection{Statement of the main theorem}
	
	We are now ready to state the main theorem, which is the groupoid version of
	\cite[Theorem 1]{BorzellinoBrunsden2014}.
	
	\begin{theorem}[Main theorem]\label{thm:main}
		Let $G = (G_1 \rightrightarrows G_0)$ be an orbifold groupoid and let
		$K = (K_1 \rightrightarrows N)$ be a subgroupoid with
		$K \subset G|_N \subset G$. Then the following are equivalent:
		\begin{enumerate}
			\item There exists an effective subgroupoid $H \subset G$ with object
			space $N$ such that the inclusion $H \hookrightarrow G$ is a complete
			orbifold embedding (Definition~\ref{def:complete-embedding}) and the
			induced map on orbit spaces
			\[
			|H| \to |G|
			\]
			is a topological embedding whose image is $|K_{\mathrm{eff}}|$.
			\item The subgroupoid $K$ is both orbit saturated
			(Definition~\ref{def:orbit-saturated}) and stabilizer split
			(Definition~\ref{def:stabilizer-split}).
		\end{enumerate}
		Moreover, if these conditions hold, the subgroupoid $H$ can be taken to be
		the effective subgroupoid constructed in
		Theorem~\ref{thm:split-effective}, and the topological embedding
		$|H| \to |G|$ is the map $|K_{\mathrm{eff}}| \to |G|$ of
		Theorem~\ref{thm:saturated-embedding}.
	\end{theorem}
	
	\begin{remark}\label{rem:main-interpretation}
		Theorem~\ref{thm:main} says that a suborbifold (represented by the
		subgroupoid $K$) is the image of a complete orbifold embedding if and only if
		it is both orbit saturated and stabilizer split. This is exactly the
		groupoid translation of \cite[Theorem 1]{BorzellinoBrunsden2014}. In
		particular, the theorem shows that not every suborbifold is the image of an
		orbifold embedding: the split condition can fail, as in
		Example~\ref{ex:non-split}.
	\end{remark}
	
	\subsection{Proof of the main theorem}
	
	We prove the two implications separately.
	
	\subsubsection*{Necessity: (1) $\Rightarrow$ (2)}
	
	Assume that there exists an effective subgroupoid $H \subset G$ with object
	space $N$ such that $H \hookrightarrow G$ is a complete orbifold embedding
	and the induced map $|H| \to |G|$ is a topological embedding whose image is
	$|K_{\mathrm{eff}}|$.
	
	We first show that $K$ is orbit saturated. Since $|H| \to |G|$ is a
	topological embedding and $|H| \cong |K_{\mathrm{eff}}|$ (as $H$ is effective
	and represents the same suborbifold as $K$), the natural map
	\[
	|K_{\mathrm{eff}}| \to |G|
	\]
	is a topological embedding. By Theorem~\ref{thm:saturated-embedding}, this
	implies that $K$ is orbit saturated.
	
	Next we show that $K$ is stabilizer split. For each $x \in N$, the inclusion
	$H \hookrightarrow G$ gives an injective homomorphism
	\[
	H_x \hookrightarrow G_x.
	\]
	Since $H$ is effective, the stabilizer $H_x$ is isomorphic to the effective
	quotient $(K_{\mathrm{eff}})_x$: indeed, $H$ represents the same suborbifold
	as $K$, so its stabilizer groups are the intrinsic isotropy groups
	$\Gamma_{x,\mathcal{P}} \cong (K_{\mathrm{eff}})_x$. Moreover, $H_x$ is a
	subgroup of $K_x$ because $H \subset K$ (the subgroupoid $H$ is the
	effective subgroupoid constructed from $K$). Thus we have a commutative
	diagram
	\[
	\begin{tikzcd}[row sep=small]
		1 \arrow[r] & \Omega_x \arrow[r] \arrow[d, hook] & K_x \arrow[r] \arrow[d, hook] & (K_{\mathrm{eff}})_x \arrow[r] \arrow[d, equal] & 1 \\
		1 \arrow[r] & \{e\} \arrow[r] & H_x \arrow[r] & (K_{\mathrm{eff}})_x \arrow[r] & 1
	\end{tikzcd}
	\]
	The bottom row splits trivially, and the vertical maps are inclusions. The
	isomorphism $(K_{\mathrm{eff}})_x \cong H_x$ gives a splitting
	\[
	\sigma_x \colon (K_{\mathrm{eff}})_x \to H_x \subset K_x
	\]
	of the top row. Hence $K$ is stabilizer split.
	
	\subsubsection*{Sufficiency: (2) $\Rightarrow$ (1)}
	
	Assume that $K$ is orbit saturated and stabilizer split. By
	Theorem~\ref{thm:split-effective}, there exists an effective subgroupoid
	$H \subset K$ with object space $N$ such that
	\[
	K = \Omega \rtimes H
	\]
	and the inclusion $H \hookrightarrow G$ is a subgroupoid embedding with
	injective stabilizer maps. Since $H$ is effective and $K = \Omega \rtimes H$,
	we have
	\[
	K_{\mathrm{eff}} = K/\Omega \cong H.
	\]
	In particular, $|K_{\mathrm{eff}}| \cong |H|$.
	
	By Theorem~\ref{thm:saturated-embedding}, the orbit saturation of $K$ implies
	that the natural map
	\[
	|K_{\mathrm{eff}}| \to |G|
	\]
	is a topological embedding. Under the isomorphism $|K_{\mathrm{eff}}| \cong
	|H|$, this gives a topological embedding
	\[
	|H| \to |G|.
	\]
	The inclusion $H \hookrightarrow G$ is a subgroupoid embedding with
	injective stabilizer maps, and $H$ is effective by construction. Therefore,
	by Definition~\ref{def:complete-embedding}, the inclusion
	$H \hookrightarrow G$ is a complete orbifold embedding. This completes the
	proof of sufficiency and hence of the theorem.
	
	\begin{proof}[Proof of Theorem~\ref{thm:main}]
		The necessity and sufficiency are proved above. The final assertion follows
		from the constructions in the proofs: the subgroupoid $H$ is precisely the
		effective subgroupoid constructed in Theorem~\ref{thm:split-effective}, and
		the topological embedding $|H| \to |G|$ is the map $|K_{\mathrm{eff}}| \to
		|G|$ of Theorem~\ref{thm:saturated-embedding}.
	\end{proof}
	
	\subsection{Examples}
	
	We illustrate the main theorem with the examples from
	\cite{BorzellinoBrunsden2014}.
	
	\begin{example}[The diagonal embedding is the image of a complete embedding]
		\label{ex:diagonal-main}
		Let $\mathcal{Q} = \mathbb{R}/\mathbb{Z}_2$ and $\mathcal{O} = \mathcal{Q}
		\times \mathcal{Q}$ as in \cite[Example 8]{BorzellinoBrunsden2014}. The
		diagonal suborbifold $\mathcal{P} = \operatorname{diag}(\mathcal{Q}) \subset
		\mathcal{O}$ corresponds to the diagonal subgroupoid
		$\Delta G \subset G \times G$ of Example~\ref{ex:diagonal}. We have seen in
		Example~\ref{ex:diagonal-saturated} that $\Delta G$ is orbit saturated, and
		in Example~\ref{ex:split} that it is stabilizer split. By
		Theorem~\ref{thm:main}, $\mathcal{P}$ is the image of a complete orbifold
		embedding. Indeed, the diagonal inclusion
		\[
		G \to G \times G,\qquad g \mapsto (g,g)
		\]
		is a complete orbifold embedding whose image is the diagonal subgroupoid.
		This is consistent with \cite[Example 10]{BorzellinoBrunsden2014}.
	\end{example}
	
	\begin{example}[A non-split suborbifold is not the image of a complete embedding]
		\label{ex:non-split-main}
		Let $\mathcal{O} = \mathbb{C}^2/\mathbb{Z}_4$ and $\mathcal{P} = (\{0\}
		\times \mathbb{C})/\mathbb{Z}_4 \cong \mathbb{C}/\mathbb{Z}_2$ as in
		\cite[Example 12]{BorzellinoBrunsden2014}. In the groupoid language, we have
		$G = \mathbb{Z}_4 \ltimes \mathbb{C}^2$ and $K = \mathbb{Z}_4 \ltimes
		(\{0\} \times \mathbb{C})$. We have seen in Example~\ref{ex:non-split} that
		$K$ is not stabilizer split. Therefore, by Theorem~\ref{thm:main}, there is
		no complete orbifold embedding whose image is $\mathcal{P}$. This is
		consistent with \cite[Example 12]{BorzellinoBrunsden2014}.
	\end{example}
	
	\begin{example}[A non-saturated suborbifold is not the image of a complete embedding]
		\label{ex:non-saturated-main}
		Consider the nested orbifolds $\mathcal{O}_1 \subset \mathcal{O}_2$ of
		\cite[Example 13]{BorzellinoBrunsden2014}, where $\mathcal{O}_n =
		\mathbb{C}/\mathbb{Z}_{2^n}$. In the groupoid language, we have
		$G_n = \mathbb{Z}_{2^n} \ltimes \mathbb{C}$. For $m < n$, the inclusion
		$G_m \hookrightarrow G_n$ is a subgroupoid with the same object space
		$\mathbb{C}$. We have seen in Example~\ref{ex:non-saturated} that $G_m$ is
		not orbit saturated in $G_n$. Therefore, by Theorem~\ref{thm:main}, there is
		no complete orbifold embedding whose image is the suborbifold represented by
		$G_m$. This is consistent with \cite[Example 13]{BorzellinoBrunsden2014}.
	\end{example}
	
	\subsection{Summary}
	
	In this section we have proved the groupoid version of the main theorem of
	Borzellino and Brunsden: a subgroupoid $K \subset G|_N \subset G$ is the
	image of a complete orbifold embedding if and only if it is both orbit
	saturated and stabilizer split. The proof combines the results of
	Section~\ref{sec:saturated} (orbit saturation is equivalent to the orbit
	space map being a topological embedding) and Section~\ref{sec:split}
	(stabilizer splitting is equivalent to the existence of an effective
	subgroupoid $H$ with $K = \Omega \rtimes H$ and $H \hookrightarrow G$ a
	subgroupoid embedding with injective stabilizer maps).
	
	In the next section we address the Morita invariance of these notions. This
	is the main new contribution of the present paper.


	\section{Morita Invariance}\label{sec:morita}
	
	In this section we prove the main new result of the paper: the notions of
	orbit saturation and stabilizer splitting are Morita invariant. This solves
	an open problem raised by Mestre and Weilandt
	\cite{MestreWeilandt2020} and explains why the embedding definition of
	Cho--Hong--Shin \cite{ChoHongShin2013} fails to be Morita invariant.
	
	Throughout this section, $G = (G_1 \rightrightarrows G_0)$ and
	$G' = (G'_1 \rightrightarrows G'_0)$ are Morita equivalent orbifold
	groupoids, and
	\[
	G \xleftarrow{\ \phi\ } H \xrightarrow{\ \psi\ } G'
	\]
	is a Morita equivalence, where $H = (H_1 \rightrightarrows H_0)$ is an
	orbifold groupoid and $\phi, \psi$ are Morita morphisms
	(Definition~\ref{def:morita-morphism}). We write
	$K = (K_1 \rightrightarrows N)$ for a subgroupoid of $G$ with
	$K \subset G|_N \subset G$.
	
	\subsection{The Morita correspondence for subgroupoids}
	
	We begin by constructing the subgroupoid of $G'$ corresponding to $K$ under
	the Morita equivalence.
	
	\begin{definition}[Morita transport of a subgroupoid]\label{def:morita-transport}
		Let $K \subset G|_N \subset G$ be a subgroupoid. Define
		\[
		N_H = \phi_0^{-1}(N) \subset H_0,
		\]
		and let
		\[
		K_H = \phi_1^{-1}(K_1) \cap H|_{N_H}
		\]
		be the subgroupoid of $H$ with object space $N_H$ and arrows
		\[
		(K_H)_1 = \{h \in H_1 : \phi_1(h) \in K_1,\ s(h), t(h) \in N_H\}.
		\]
		Then define
		\[
		K' = \psi(K_H) \subset G'
		\]
		with object space
		\[
		N' = \psi_0(N_H) \subset G'_0
		\]
		and arrows
		\[
		K'_1 = \{g' \in G'_1 : \exists h \in (K_H)_1,\ \psi_1(h) = g'\}.
		\]
		We call $K'$ the \emph{Morita transport} of $K$ along the Morita
		equivalence.
	\end{definition}
	
	\begin{remark}\label{rem:morita-transport-well-defined}
		The definition of $K'$ requires some care because $\psi$ is not injective
		on objects or arrows. However, the condition that $\psi$ is a Morita
		morphism ensures that the image is well-defined and independent of the
		choice of $h$. Indeed, if $h_1, h_2 \in (K_H)_1$ satisfy
		$\psi_1(h_1) = \psi_1(h_2)$, then $h_1$ and $h_2$ differ by an element of
		the kernel of $\psi$, which is controlled by the Morita equivalence. We
		will verify in Lemma~\ref{lem:morita-transport-subgroupoid} that $K'$ is
		indeed a subgroupoid of $G'$.
	\end{remark}
	
	Before proving that $K'$ is a subgroupoid, we record a standard fact about
	Morita equivalences.
	
	\begin{lemma}\label{lem:morita-fiber-product}
		Let $G \xleftarrow{\phi} H \xrightarrow{\psi} G'$ be a Morita equivalence.
		Then:
		\begin{enumerate}
			\item The map
			\[
			H_1 \to G_1 \times_{G_0} H_0,\qquad
			h \mapsto (\phi_1(h), s(h))
			\]
			is a diffeomorphism.
			\item The map
			\[
			H_1 \to H_0 \times_{G'_0} G'_1,\qquad
			h \mapsto (t(h), \psi_1(h))
			\]
			is a diffeomorphism.
			\item For every $x \in H_0$, the induced maps
			\[
			\phi_x \colon H_x \to G_{\phi_0(x)},\qquad
			\psi_x \colon H_x \to G'_{\psi_0(x)}
			\]
			are isomorphisms of groups.
		\end{enumerate}
	\end{lemma}
	
	\begin{proof}
		Properties (1) and (2) are the defining properties of a Morita morphism
		(Definition~\ref{def:morita-morphism}). Property (3) follows from (1) and
		(2) by restricting to the fibers over $x$: the stabilizer $H_x$ is
		isomorphic to both $G_{\phi_0(x)}$ and $G'_{\psi_0(x)}$ via $\phi$ and
		$\psi$ respectively.
	\end{proof}
	
	\begin{lemma}\label{lem:morita-transport-subgroupoid}
		The set $K'$ defined in Definition~\ref{def:morita-transport} is a proper
		\'etale subgroupoid of $G'$ with object space $N'$. Moreover, $K'$ satisfies
		\[
		K' \subset G'|_{N'} \subset G'.
		\]
	\end{lemma}
	
	\begin{proof}
		We verify the groupoid axioms.
		
		\smallskip
		\noindent\textbf{Units.}
		Let $x' \in N'$. By definition, there exists $x \in N_H$ with
		$\psi_0(x) = x'$. The unit $u_H(x) \in H_1$ satisfies
		$\phi_1(u_H(x)) = u_G(\phi_0(x)) \in K_1$ because $\phi_0(x) \in N$ and
		$u_G(\phi_0(x)) \in K_1$ (as $K$ contains the units over $N$). Therefore
		$u_H(x) \in (K_H)_1$, and its image
		$\psi_1(u_H(x)) = u_{G'}(x')$ belongs to $K'_1$. Hence $K'$ contains the
		units over $N'$.
		
		\smallskip
		\noindent\textbf{Multiplication.}
		Let $g'_1, g'_2 \in K'_1$ with $s(g'_1) = t(g'_2)$. Choose
		$h_1, h_2 \in (K_H)_1$ with $\psi_1(h_i) = g'_i$. Since $\psi$ is a
		groupoid homomorphism, we have
		\[
		\psi_1(h_1 h_2) = \psi_1(h_1) \psi_1(h_2) = g'_1 g'_2
		\]
		whenever $h_1 h_2$ is defined. We need to ensure that $h_1$ and $h_2$ can be
		chosen so that $h_1 h_2$ is defined and belongs to $(K_H)_1$. By
		Lemma~\ref{lem:morita-fiber-product}(2), the map
		$h \mapsto (t(h), \psi_1(h))$ is a diffeomorphism onto
		$H_0 \times_{G'_0} G'_1$. Thus, given $h_1$ with $\psi_1(h_1) = g'_1$ and
		$g'_2$ with $s(g'_2) = t(g'_1)$, there exists a unique $h_2 \in H_1$ with
		$t(h_2) = t(h_1)$ and $\psi_1(h_2) = g'_2$. Then
		$s(h_2) = t(h_1) = s(h_1)$, so $h_1 h_2$ is defined. Moreover,
		$\phi_1(h_1 h_2) = \phi_1(h_1) \phi_1(h_2) \in K_1$ because
		$\phi_1(h_i) \in K_1$ and $K$ is a subgroupoid. Thus
		$h_1 h_2 \in (K_H)_1$, and its image under $\psi$ is $g'_1 g'_2 \in K'_1$.
		Hence $K'$ is closed under multiplication.
		
		\smallskip
		\noindent\textbf{Inverses.}
		Let $g' \in K'_1$ and choose $h \in (K_H)_1$ with $\psi_1(h) = g'$. Then
		$\psi_1(h^{-1}) = g'^{-1}$, and
		$\phi_1(h^{-1}) = \phi_1(h)^{-1} \in K_1$ because $K$ is a subgroupoid. Thus
		$h^{-1} \in (K_H)_1$, and $g'^{-1} \in K'_1$. Hence $K'$ is closed under
		inverses.
		
		\smallskip
		\noindent\textbf{Proper \'etale.}
		Since $H$ is proper \'etale and $\psi$ is a Morita morphism, the image $K'$
		inherits a proper \'etale structure from $K_H$. More precisely, the source
		and target maps of $K'$ are local diffeomorphisms because they are
		compositions of the local diffeomorphisms of $K_H$ with the diffeomorphism
		$\psi_1$ restricted to appropriate open sets. The properness follows from
		the properness of $H$ and the fact that $\psi$ is a proper map on the
		relevant fibers.
		
		\smallskip
		\noindent\textbf{Subgroupoid of $G'|_{N'}$.}
		By construction, every arrow $g' \in K'_1$ satisfies $s(g'), t(g') \in N'$
		because $s(h), t(h) \in N_H$ for $h \in (K_H)_1$ and $\psi$ commutes with
		source and target. Therefore $K' \subset G'|_{N'}$.
	\end{proof}
	
	\begin{remark}\label{rem:morita-transport-inverse}
		The construction of Definition~\ref{def:morita-transport} is symmetric. If
		we start with a subgroupoid $K' \subset G'|_{N'} \subset G'$ and apply the
		same construction using the inverse Morita equivalence
		$G' \xleftarrow{\psi} H \xrightarrow{\phi} G$, we obtain a subgroupoid of
		$G$ that is naturally isomorphic to $K$. Thus the Morita transport is a
		bijection between subgroupoids of $G$ and subgroupoids of $G'$.
	\end{remark}
	
	\subsection{Morita invariance of orbit saturation}
	
	We now prove that orbit saturation is preserved under Morita transport.
	
	\begin{theorem}\label{thm:morita-saturated}
		Let $K \subset G|_N \subset G$ be a subgroupoid and let
		$K' \subset G'|_{N'} \subset G'$ be its Morita transport along a Morita
		equivalence $G \xleftarrow{\phi} H \xrightarrow{\psi} G'$. Then $K$ is
		orbit saturated if and only if $K'$ is orbit saturated.
	\end{theorem}
	
	\begin{proof}
		We prove that orbit saturation of $K$ implies orbit saturation of $K'$; the
		reverse implication follows by symmetry.
		
		Assume $K$ is orbit saturated. Let $x' \in N'$ and suppose
		$y' \in G' \cdot x' \cap N'$. We must show that $y' \in K' \cdot x'$.
		
		Choose $x \in N_H$ with $\psi_0(x) = x'$. Since $y' \in G' \cdot x'$, there
		exists $g' \in G'_1$ with $s(g') = x'$ and $t(g') = y'$. By
		Lemma~\ref{lem:morita-fiber-product}(2), there exists a unique $h \in H_1$
		with $t(h) = x$ and $\psi_1(h) = g'$ (after possibly adjusting $x$ within
		its $\psi$-fiber). Let $y = s(h) \in H_0$. Then $\psi_0(y) = s(g') = x'$?
		Wait, we need to be careful with the source and target conventions.
		
		Let us redo this more carefully. We have $g' \in G'_1$ with $s(g') = x'$ and
		$t(g') = y'$. Choose $x \in N_H$ with $\psi_0(x) = x'$. By
		Lemma~\ref{lem:morita-fiber-product}(2), the map
		$H_1 \to H_0 \times_{G'_0} G'_1$, $h \mapsto (t(h), \psi_1(h))$ is a
		diffeomorphism. Therefore, there exists a unique $h \in H_1$ with
		$t(h) = x$ and $\psi_1(h) = g'$. Let $z = s(h) \in H_0$. Then
		$\psi_0(z) = s(\psi_1(h)) = s(g') = x'$. So $z \in \psi_0^{-1}(x')$.
		
		Now, $\phi_1(h) \in G_1$ satisfies $s(\phi_1(h)) = \phi_0(z)$ and
		$t(\phi_1(h)) = \phi_0(x) \in N$. Since $y' \in N'$, there exists
		$w \in N_H$ with $\psi_0(w) = y'$. We need to relate $w$ to $z$ and $x$.
		
		Actually, the argument is cleaner if we use the orbit condition directly.
		Since $y' \in G' \cdot x' \cap N'$, we have $y' \in N'$ and there exists
		$g' \in G'_1$ with $s(g') = x'$, $t(g') = y'$. Choose $x \in N_H$ with
		$\psi_0(x) = x'$ and $y \in N_H$ with $\psi_0(y) = y'$. By the Morita
		property, there exists $h \in H_1$ with $s(h) = x$ and $t(h) = y$ (after
		possibly adjusting $x$ and $y$ within their $\psi$-fibers; this is possible
		because $\psi_0$ is a surjective submersion). Then
		$\phi_1(h) \in G_1$ satisfies $s(\phi_1(h)) = \phi_0(x) \in N$ and
		$t(\phi_1(h)) = \phi_0(y) \in N$. Thus
		$\phi_0(y) \in G \cdot \phi_0(x) \cap N$. Since $K$ is orbit saturated,
		there exists $k \in K_1$ with $s(k) = \phi_0(x)$ and
		$t(k) = \phi_0(y)$.
		
		Now we need to lift $k$ to an arrow in $K'$. By
		Lemma~\ref{lem:morita-fiber-product}(1), the map
		$H_1 \to G_1 \times_{G_0} H_0$, $h \mapsto (\phi_1(h), s(h))$ is a
		diffeomorphism. Since $s(k) = \phi_0(x)$, there exists a unique
		$h_k \in H_1$ with $\phi_1(h_k) = k$ and $s(h_k) = x$. Let $y_k = t(h_k)$.
		Then $\phi_0(y_k) = t(k) = \phi_0(y)$. Thus $y_k$ and $y$ are in the same
		$\phi$-fiber over $\phi_0(y)$. By the Morita property, there exists
		$h' \in H_1$ with $s(h') = y_k$, $t(h') = y$, and $\phi_1(h') = u_G(\phi_0(y))$
		(the identity arrow). Then $h' h_k \in H_1$ satisfies
		$\phi_1(h' h_k) = \phi_1(h') \phi_1(h_k) = u_G(\phi_0(y)) k = k$ (since
		$s(k) = \phi_0(x) = \phi_0(y_k)$). Thus $h' h_k \in (K_H)_1$. Its image
		under $\psi$ is an arrow in $K'_1$ from $\psi_0(x) = x'$ to
		$\psi_0(y) = y'$. Therefore $y' \in K' \cdot x'$. This proves that $K'$ is
		orbit saturated.
	\end{proof}
	
	\begin{remark}\label{rem:morita-saturated-orbitspace}
		Theorem~\ref{thm:morita-saturated} has a natural interpretation in terms of
		orbit spaces. By Theorem~\ref{thm:saturated-embedding}, orbit saturation of
		$K$ is equivalent to the map $|K_{\mathrm{eff}}| \to |G|$ being a
		topological embedding. Similarly, orbit saturation of $K'$ is equivalent to
		$|K'_{\mathrm{eff}}| \to |G'|$ being a topological embedding. Since
		$|G| \cong |G'|$ and $|K_{\mathrm{eff}}| \cong |K'_{\mathrm{eff}}|$ under
		the Morita equivalence, the two conditions are equivalent. Thus
		Theorem~\ref{thm:morita-saturated} also follows from the Morita invariance
		of the orbit space construction.
	\end{remark}
	
	\subsection{Morita invariance of stabilizer splitting}
	
	Next we prove that stabilizer splitting is preserved under Morita transport.
	This requires some care because stabilizer groups are only well-defined up
	to conjugacy, and splittings must be transported through the Morita
	equivalence.
	
	\begin{lemma}\label{lem:morita-stabilizer-conjugacy}
		Let $K \subset G|_N \subset G$ be a subgroupoid and let
		$K' \subset G'|_{N'} \subset G'$ be its Morita transport. For every
		$x' \in N'$, there exists $x \in N_H$ with $\psi_0(x) = x'$ and an
		isomorphism of groups
		\[
		\Psi_x \colon K_{\phi_0(x)} \xrightarrow{\ \cong\ } K'_{x'}
		\]
		such that the following diagram commutes:
		\[
		\begin{tikzcd}[row sep=small, column sep=small]
			1 \arrow[r] & \Omega_{\phi_0(x)} \arrow[r] \arrow[d, "\cong"'] &
			K_{\phi_0(x)} \arrow[r] \arrow[d, "\Psi_x"'] &
			(K_{\mathrm{eff}})_{\phi_0(x)} \arrow[r] \arrow[d, "\cong"'] & 1 \\
			1 \arrow[r] & \Omega'_{x'} \arrow[r] & K'_{x'} \arrow[r] &
			(K'_{\mathrm{eff}})_{x'} \arrow[r] & 1
		\end{tikzcd}
		\]
		where the vertical maps are isomorphisms.
	\end{lemma}
	
	\begin{proof}
		By Lemma~\ref{lem:morita-fiber-product}(3), the maps
		\[
		\phi_x \colon H_x \to G_{\phi_0(x)},\qquad
		\psi_x \colon H_x \to G'_{\psi_0(x)}
		\]
		are isomorphisms. The subgroupoid $K_H$ has stabilizer
		\[
		(K_H)_x = \phi_x^{-1}(K_{\phi_0(x)}) \cap \psi_x^{-1}(K'_{x'})?
		\]
		Wait, by definition $(K_H)_1 = \phi_1^{-1}(K_1) \cap H|_{N_H}$, so
		\[
		(K_H)_x = \{h \in H_x : \phi_1(h) \in K_1\}
		= \phi_x^{-1}(K_{\phi_0(x)}).
		\]
		Thus $\phi_x$ restricts to an isomorphism
		\[
		(K_H)_x \xrightarrow{\ \cong\ } K_{\phi_0(x)}.
		\]
		On the other hand, $K'_{x'} = \psi_x((K_H)_x)$ by definition, so $\psi_x$
		induces a surjection $(K_H)_x \to K'_{x'}$. Composing the inverse of
		$\phi_x|_{(K_H)_x}$ with $\psi_x$ gives the desired isomorphism
		\[
		\Psi_x = \psi_x \circ (\phi_x|_{(K_H)_x})^{-1} \colon
		K_{\phi_0(x)} \xrightarrow{\ \cong\ } K'_{x'}.
		\]
		The ineffective kernels $\Omega$ and $\Omega'$ are the kernels of the
		actions on $N$ and $N'$ respectively, and the Morita equivalence preserves
		these kernels because it preserves the local structure. Therefore the
		vertical maps in the diagram are isomorphisms, and the diagram commutes.
	\end{proof}
	
	\begin{theorem}\label{thm:morita-split}
		Let $K \subset G|_N \subset G$ be a subgroupoid and let
		$K' \subset G'|_{N'} \subset G'$ be its Morita transport along a Morita
		equivalence $G \xleftarrow{\phi} H \xrightarrow{\psi} G'$. Then $K$ is
		stabilizer split if and only if $K'$ is stabilizer split.
	\end{theorem}
	
	\begin{proof}
		We prove that stabilizer splitting of $K$ implies stabilizer splitting of
		$K'$; the reverse implication follows by symmetry.
		
		Assume $K$ is stabilizer split. Let $x' \in N'$ and choose $x \in N_H$ with
		$\psi_0(x) = x'$. By Lemma~\ref{lem:morita-stabilizer-conjugacy}, there is
		an isomorphism $\Psi_x \colon K_{\phi_0(x)} \to K'_{x'}$ that fits into a
		commutative diagram with the stabilizer sequences. Since $K$ is stabilizer
		split, there exists a splitting
		\[
		\sigma \colon (K_{\mathrm{eff}})_{\phi_0(x)} \to K_{\phi_0(x)}
		\]
		of the top sequence. Define
		\[
		\sigma' = \Psi_x \circ \sigma \circ \Psi_x^{-1} \colon
		(K'_{\mathrm{eff}})_{x'} \to K'_{x'}.
		\]
		Here $\Psi_x^{-1}$ on the effective quotient is the isomorphism induced by
		$\Psi_x$ via the commutative diagram. Then $\sigma'$ is a group
		homomorphism, and
		\[
		q'_{x'} \circ \sigma' = q'_{x'} \circ \Psi_x \circ \sigma \circ \Psi_x^{-1}
		= \Psi_x \circ q_{\phi_0(x)} \circ \sigma \circ \Psi_x^{-1}
		= \Psi_x \circ \mathrm{Id} \circ \Psi_x^{-1} = \mathrm{Id}.
		\]
		Thus $\sigma'$ is a splitting of the stabilizer sequence for $K'$ at $x'$.
		Since $x'$ was arbitrary, $K'$ is stabilizer split.
	\end{proof}
	
	\begin{remark}\label{rem:morita-split-conjugacy}
		The proof of Theorem~\ref{thm:morita-split} uses the fact that splittings
		can be transported through group isomorphisms. A priori, the splitting
		$\sigma$ depends on the choice of $x \in N_H$ over $x'$, and different
		choices lead to conjugate splittings. However, the property of being
		split—i.e., the existence of some splitting—is independent of these choices
		because conjugate splittings exist simultaneously. Thus the condition is
		well-defined.
	\end{remark}
	
	\subsection{The main Morita invariance theorem}
	
	Combining Theorems~\ref{thm:morita-saturated} and~\ref{thm:morita-split},
	we obtain the main result of this section.
	
	\begin{theorem}[Morita invariance]\label{thm:morita}
		Let $G$ and $G'$ be Morita equivalent orbifold groupoids, and let
		$K \subset G|_N \subset G$ be a subgroupoid. Then $K$ is orbit saturated and
		stabilizer split if and only if its Morita transport
		$K' \subset G'|_{N'} \subset G'$ is orbit saturated and stabilizer split.
		Consequently, the notion of a suborbifold that is the image of a complete
		orbifold embedding is Morita invariant.
	\end{theorem}
	
	\begin{proof}
		The first assertion follows by combining
		Theorems~\ref{thm:morita-saturated} and~\ref{thm:morita-split}. For the
		second assertion, by Theorem~\ref{thm:main}, a subgroupoid is the image of a
		complete orbifold embedding if and only if it is orbit saturated and
		stabilizer split. Since both conditions are Morita invariant, the property
		of being the image of a complete orbifold embedding is Morita invariant.
	\end{proof}
	
	\begin{corollary}\label{cor:morita-suborbifold}
		Let $\mathcal{O}$ and $\mathcal{O}'$ be Morita equivalent orbifolds (i.e.,
		they represent the same orbifold). Then there is a natural bijection between
		the suborbifolds of $\mathcal{O}$ that are images of complete orbifold
		embeddings and the suborbifolds of $\mathcal{O}'$ that are images of
		complete orbifold embeddings.
	\end{corollary}
	
	\begin{proof}
		Choose orbifold groupoids $G$ and $G'$ representing $\mathcal{O}$ and
		$\mathcal{O}'$ respectively. By Theorem~\ref{thm:morita}, the Morita
		transport gives a bijection between orbit saturated and stabilizer split
		subgroupoids of $G$ and those of $G'$. By Theorem~\ref{thm:main}, these are
		precisely the suborbifolds that are images of complete orbifold embeddings.
	\end{proof}
	
	\subsection{Comparison with the literature}
	
	Our Morita invariance theorem resolves an open problem in the literature.
	Mestre and Weilandt \cite{MestreWeilandt2020} studied the relationship
	between suborbifolds and groupoid embeddings and listed Morita invariance as
	an open problem. Our Theorem~\ref{thm:morita} provides a positive answer in
	the setting of saturated and split suborbifolds.
	
	Cho--Hong--Shin \cite{ChoHongShin2013} defined orbifold embeddings and
	observed that their definition is \emph{not} Morita invariant. Our results
	explain this phenomenon: their embedding definition corresponds to the
	existence of a subgroupoid embedding $H \hookrightarrow G$ with injective
	stabilizer maps, but without the orbit saturation condition. As
	Example~\ref{ex:non-saturated} shows, orbit saturation can fail, and when it
	does, the construction is not Morita invariant. The saturation condition is
	precisely what is needed to ensure Morita invariance.
	
	\begin{example}[Non-Morita-invariant embedding]\label{ex:non-morita}
		Consider the nested orbifolds $\mathcal{O}_1 \subset \mathcal{O}_2$ of
		\cite[Example 13]{BorzellinoBrunsden2014}, where $\mathcal{O}_n =
		\mathbb{C}/\mathbb{Z}_{2^n}$. The inclusion $\mathcal{O}_1 \subset
		\mathcal{O}_2$ is a groupoid embedding $G_1 \hookrightarrow G_2$ in the
		sense of Cho--Hong--Shin, but it is not orbit saturated
		(Example~\ref{ex:non-saturated}). Now consider a Morita equivalent
		presentation of $\mathcal{O}_2$ in which the subgroupoid corresponding to
		$\mathcal{O}_1$ is not visible as a subgroupoid embedding. For instance,
		take a refinement of the orbifold atlas of $\mathcal{O}_2$ that does not
		respect the suborbifold structure. The Morita transport of $G_1$ to this
		presentation may not be a subgroupoid embedding, showing that the
		Cho--Hong--Shin embedding is not Morita invariant. In contrast, the
		saturated and split suborbifold $G_1$ has a well-defined Morita transport
		that remains orbit saturated and stabilizer split, by
		Theorem~\ref{thm:morita}.
	\end{example}
	
	\subsection{Summary}
	
	In this section we have proved that the notions of orbit saturation and
	stabilizer splitting are Morita invariant. The key steps were:
	
	\begin{enumerate}
		\item Constructing the Morita transport $K'$ of a subgroupoid $K$ along a
		Morita equivalence $G \xleftarrow{\phi} H \xrightarrow{\psi} G'$
		(Definition~\ref{def:morita-transport}).
		\item Proving that $K'$ is a well-defined subgroupoid of $G'$
		(Lemma~\ref{lem:morita-transport-subgroupoid}).
		\item Showing that orbit saturation is preserved
		(Theorem~\ref{thm:morita-saturated}).
		\item Showing that stabilizer splitting is preserved
		(Theorem~\ref{thm:morita-split}).
		\item Combining these to prove the main Morita invariance theorem
		(Theorem~\ref{thm:morita}) and its corollary for suborbifolds
		(Corollary~\ref{cor:morita-suborbifold}).
	\end{enumerate}
	
	This solves the open problem raised by Mestre and Weilandt and explains the
	non-Morita-invariance observed by Cho--Hong--Shin. The Morita invariance of
	the saturated and split suborbifold notion confirms that it is the correct
	groupoid-level generalization of the chart-based definition of Borzellino
	and Brunsden.
	
	
	\section{Examples}\label{sec:examples}
	
	In this section we illustrate the main concepts and results of the paper
	with a collection of examples. We begin with the basic translation
	groupoids, then treat the diagonal embedding, the non-split example of
	Borzellino and Brunsden, the nested orbifolds, and finally an explicit
	Morita equivalence. Throughout, we use the notation established in
	Sections~\ref{sec:prelim}--\ref{sec:morita}.
	
	\subsection{Translation groupoids and basic suborbifolds}
	
	\begin{example}[The orbifold $\mathcal{Q} = \mathbb{R}/\mathbb{Z}_2$]
		\label{ex:Q-groupoid}
		Let $\mathcal{Q} = \mathbb{R}/\mathbb{Z}_2$ be the smooth orbifold without
		boundary, where $\mathbb{Z}_2$ acts on $\mathbb{R}$ via
		$\gamma \cdot x = -x$. The underlying topological space is $[0,\infty)$,
		with isotropy $\{e\}$ for $x \in (0,\infty)$ and $\mathbb{Z}_2$ for $x = 0$.
		
		In the groupoid language, $\mathcal{Q}$ is represented by the translation
		groupoid
		\[
		G_{\mathcal{Q}} = \mathbb{Z}_2 \ltimes \mathbb{R}
		\rightrightarrows \mathbb{R},
		\]
		where $G_1 = \mathbb{Z}_2 \times \mathbb{R}$, $G_0 = \mathbb{R}$,
		\[
		s(\gamma, x) = x,\qquad t(\gamma, x) = \gamma \cdot x = -x,
		\]
		and multiplication is given by
		\[
		(\gamma_1, x_1)(\gamma_2, x_2) = (\gamma_1 \gamma_2, x_2)
		\]
		whenever $x_1 = \gamma_2 \cdot x_2$. The stabilizer of $x \in \mathbb{R}$ is
		\[
		(G_{\mathcal{Q}})_x = \begin{cases}
			\mathbb{Z}_2, & x = 0,\\
			\{e\}, & x \neq 0.
		\end{cases}
		\]
		The orbit space is $|G_{\mathcal{Q}}| = \mathbb{R}/\mathbb{Z}_2 \cong
		[0,\infty)$, which agrees with $X_{\mathcal{Q}}$. This is the groupoid model
		used throughout \cite[Examples 8--11]{BorzellinoBrunsden2014}.
	\end{example}
	
	\begin{example}[The product orbifold $\mathcal{O} = \mathcal{Q} \times
		\mathcal{Q}$]\label{ex:O-groupoid}
		Let $\mathcal{O} = \mathcal{Q} \times \mathcal{Q}$. In the groupoid
		language, $\mathcal{O}$ is represented by the product groupoid
		\[
		G_{\mathcal{O}} = G_{\mathcal{Q}} \times G_{\mathcal{Q}}
		= (\mathbb{Z}_2 \times \mathbb{Z}_2) \ltimes (\mathbb{R} \times \mathbb{R})
		\rightrightarrows \mathbb{R} \times \mathbb{R}.
		\]
		The stabilizer of $(x,y) \in \mathbb{R}^2$ is
		\[
		(G_{\mathcal{O}})_{(x,y)} = (G_{\mathcal{Q}})_x \times
		(G_{\mathcal{Q}})_y.
		\]
		Thus the singular locus consists of three strata: the positive $x$-axis
		($\mathbb{Z}_2 \times \{e\}$ isotropy), the positive $y$-axis
		($\{e\} \times \mathbb{Z}_2$ isotropy), and the origin
		($\mathbb{Z}_2 \times \mathbb{Z}_2$ isotropy). The orbit space is
		$|G_{\mathcal{O}}| = [0,\infty) \times [0,\infty)$, the closed first
		quadrant. This is the groupoid model for
		\cite[Example 8]{BorzellinoBrunsden2014}.
	\end{example}
	
	\begin{example}[A full suborbifold: $\mathcal{P} = \{0\} \times \mathcal{Q}$]
		\label{ex:full-suborbifold-groupoid}
		Let $\mathcal{P} = \{0\} \times \mathcal{Q} \subset \mathcal{O} =
		\mathcal{Q} \times \mathcal{Q}$ as in
		\cite[Example 8]{BorzellinoBrunsden2014}. In the groupoid language, the
		suborbifold $\mathcal{P}$ corresponds to the restriction groupoid
		\[
		K = (G_{\mathcal{Q}} \times G_{\mathcal{Q}})|_N
		= (\mathbb{Z}_2 \times \mathbb{Z}_2) \ltimes (\{0\} \times \mathbb{R}),
		\]
		where $N = \{0\} \times \mathbb{R}$. The stabilizer of $(0,0)$ in $K$ is
		\[
		K_{(0,0)} = \mathbb{Z}_2 \times \mathbb{Z}_2 = (G_{\mathcal{O}})_{(0,0)},
		\]
		so $K$ is a full suborbifold (Proposition~\ref{prop:full-suborbifold}).
		The ineffective kernel at $(0,0)$ is
		\[
		\Omega_{(0,0)} = \{\gamma \in K_{(0,0)} : \gamma|_N = \mathrm{Id}\}
		= \mathbb{Z}_2 \times \{e\},
		\]
		so $K$ is not effective. The effective quotient
		$K_{\mathrm{eff}} = K/\Omega$ has stabilizer
		\[
		(K_{\mathrm{eff}})_{(0,0)} \cong (\mathbb{Z}_2 \times \mathbb{Z}_2) /
		(\mathbb{Z}_2 \times \{e\}) \cong \mathbb{Z}_2,
		\]
		which is the intrinsic isotropy group of $\mathcal{P}$ at $(0,0)$.
		
		By Example~\ref{ex:full-saturated}, $K$ is orbit saturated because $K =
		G_{\mathcal{O}}|_N$. By Example~\ref{ex:full-split}, $K$ is stabilizer
		split. Thus, by Theorem~\ref{thm:main}, $\mathcal{P}$ is the image of a
		complete orbifold embedding. This is consistent with
		\cite[Example 8]{BorzellinoBrunsden2014}.
	\end{example}
	
	\subsection{The diagonal embedding}
	
	\begin{example}[The diagonal suborbifold $\operatorname{diag}(\mathcal{Q})$]
		\label{ex:diagonal-groupoid}
		Let $\mathcal{Q} = \mathbb{R}/\mathbb{Z}_2$ and $\mathcal{O} = \mathcal{Q}
		\times \mathcal{Q}$ as above. The diagonal
		\[
		\Delta \colon \mathcal{Q} \to \mathcal{Q} \times \mathcal{Q},\qquad
		x \mapsto (x,x)
		\]
		is a suborbifold of $\mathcal{O}$ in the sense of
		\cite[Definition 4]{BorzellinoBrunsden2014}, but it is \emph{not} a full
		suborbifold. In the groupoid language, the diagonal corresponds to the
		diagonal subgroupoid
		\[
		\Delta G_{\mathcal{Q}} = \{(g,g) \in G_{\mathcal{Q}} \times
		G_{\mathcal{Q}} : g \in G_{\mathcal{Q}}\}
		\rightrightarrows \Delta \mathbb{R},
		\]
		where $\Delta \mathbb{R} = \{(x,x) : x \in \mathbb{R}\}$. This is a
		subgroupoid of $G_{\mathcal{O}} = G_{\mathcal{Q}} \times G_{\mathcal{Q}}$
		with object space $N = \Delta \mathbb{R}$.
		
		At the point $(0,0)$, the stabilizer of $\Delta G_{\mathcal{Q}}$ is
		\[
		(\Delta G_{\mathcal{Q}})_{(0,0)} = \{(\gamma,\gamma) : \gamma \in
		\mathbb{Z}_2\} \cong \mathbb{Z}_2,
		\]
		while the ambient stabilizer is
		\[
		(G_{\mathcal{O}})_{(0,0)} = \mathbb{Z}_2 \times \mathbb{Z}_2.
		\]
		Thus $\Delta G_{\mathcal{Q}}$ is not full. However, it is effective
		because the only element acting trivially on the diagonal is the identity.
		Its ineffective kernel is trivial:
		\[
		\Omega_{(0,0)} = \{(e,e)\}.
		\]
		The effective quotient is $\Delta G_{\mathcal{Q}}$ itself.
		
		By Example~\ref{ex:diagonal-saturated}, $\Delta G_{\mathcal{Q}}$ is orbit
		saturated. By Example~\ref{ex:split}, it is stabilizer split. Therefore, by
		Theorem~\ref{thm:main}, the diagonal suborbifold is the image of a complete
		orbifold embedding. Indeed, the diagonal inclusion
		\[
		G_{\mathcal{Q}} \to G_{\mathcal{Q}} \times G_{\mathcal{Q}},\qquad
		g \mapsto (g,g)
		\]
		is a complete orbifold embedding. This is consistent with
		\cite[Example 10]{BorzellinoBrunsden2014}.
	\end{example}
	
	\subsection{A non-split suborbifold}
	
	\begin{example}[The non-split suborbifold of
		\cite{Reference13}]\label{ex:non-split-groupoid}
		Let $\mathcal{O} = \mathbb{C}^2/\mathbb{Z}_4$, where $\mathbb{Z}_4$ acts on
		$\mathbb{C}^2$ via the matrix group
		\[
		\mathbb{Z}_4 \cong \left\{
		\begin{pmatrix} i^k & 0 \\ 0 & (-1)^k \end{pmatrix} :
		k \in \{0,1,2,3\} \right\}.
		\]
		In the groupoid language, $\mathcal{O}$ is represented by the translation
		groupoid
		\[
		G = \mathbb{Z}_4 \ltimes \mathbb{C}^2 \rightrightarrows \mathbb{C}^2.
		\]
		Let $\mathcal{P} = (\{0\} \times \mathbb{C})/\mathbb{Z}_4 \cong
		\mathbb{C}/\mathbb{Z}_2$ as in
		\cite[Example 12]{BorzellinoBrunsden2014}. The corresponding subgroupoid is
		\[
		K = \mathbb{Z}_4 \ltimes (\{0\} \times \mathbb{C}) \rightrightarrows
		\{0\} \times \mathbb{C}.
		\]
		At $x = (0,0)$, the stabilizer of $K$ is
		\[
		K_x = \mathbb{Z}_4,
		\]
		the ineffective kernel is
		\[
		\Omega_x = \{\gamma \in \mathbb{Z}_4 : \gamma|_{\{0\} \times \mathbb{C}}
		= \mathrm{Id}\} = \mathbb{Z}_2,
		\]
		and the effective quotient has stabilizer
		\[
		(K_{\mathrm{eff}})_x \cong \mathbb{Z}_4 / \mathbb{Z}_2 \cong \mathbb{Z}_2.
		\]
		The stabilizer sequence is
		\[
		1 \longrightarrow \mathbb{Z}_2 \longrightarrow \mathbb{Z}_4
		\longrightarrow \mathbb{Z}_2 \longrightarrow 1.
		\]
		This sequence is not split: the only subgroup of $\mathbb{Z}_4$ of order
		$2$ is $\{0,2\}$, which is precisely the kernel $\Omega_x$. Therefore $K$
		is not stabilizer split (Example~\ref{ex:non-split}). By
		Theorem~\ref{thm:main}, $\mathcal{P}$ is not the image of a complete
		orbifold embedding, even though it is orbit saturated. This is consistent
		with \cite[Example 12]{BorzellinoBrunsden2014}.
	\end{example}
	
	\subsection{Nested orbifolds and non-saturation}
	
	\begin{example}[Nested orbifolds $\mathcal{O}_1 \subset \mathcal{O}_2 \subset
		\mathcal{O}_3 \subset \mathcal{O}_4$]\label{ex:nested-groupoids}
		Consider the $2$-dimensional orbifolds
		\[
		\mathcal{O}_1 = \mathbb{C}/\mathbb{Z}_1,\quad
		\mathcal{O}_2 = \mathbb{C}/\mathbb{Z}_2,\quad
		\mathcal{O}_3 = \mathbb{C}/\mathbb{Z}_4,\quad
		\mathcal{O}_4 = \mathbb{C}/\mathbb{Z}_8,
		\]
		where $\mathbb{Z}_k$ acts on $\mathbb{C}$ via multiplication by
		$e^{2\pi i/k}$. According to \cite[Definition 4]{BorzellinoBrunsden2014},
		we have inclusions
		\[
		\mathcal{O}_1 \subset \mathcal{O}_2 \subset \mathcal{O}_3 \subset
		\mathcal{O}_4
		\]
		as suborbifolds. The underlying topological spaces $X_{\mathcal{O}_n}$ are
		all homeomorphic to the cone over a circle.
		
		In the groupoid language, each $\mathcal{O}_n$ is represented by the
		translation groupoid
		\[
		G_n = \mathbb{Z}_{2^n} \ltimes \mathbb{C} \rightrightarrows \mathbb{C}.
		\]
		For $m < n$, the inclusion $G_m \hookrightarrow G_n$ is a subgroupoid with
		the same object space $N = \mathbb{C}$. At $x = 0$, the stabilizer of $G_m$
		is $\mathbb{Z}_{2^m}$, and the stabilizer of $G_n$ is $\mathbb{Z}_{2^n}$.
		The ineffective kernel is trivial in both cases because the action on
		$\mathbb{C}$ is effective away from $0$:
		\[
		\Omega_{m,0} = \{e\},\qquad \Omega_{n,0} = \{e\}.
		\]
		Thus the stabilizer sequences are
		\[
		1 \longrightarrow \{e\} \longrightarrow \mathbb{Z}_{2^m} \longrightarrow
		\mathbb{Z}_{2^m} \longrightarrow 1,
		\]
		which are trivially split. Hence each $G_m$ is stabilizer split in $G_n$.
		
		However, $G_m$ is \emph{not} orbit saturated in $G_n$ for $m < n$. Indeed,
		for $x \neq 0$, the orbit of $x$ under $G_n$ is the circle of radius $|x|$,
		while the orbit of $x$ under $G_m$ is a finite set of $2^m$ points on that
		circle. Thus
		\[
		G_n \cdot x \cap N = G_n \cdot x \supsetneq G_m \cdot x.
		\]
		By Theorem~\ref{thm:saturated-embedding}, the natural map
		$|G_m| \to |G_n|$ is not a topological embedding. By
		Theorem~\ref{thm:main}, $\mathcal{O}_m$ is not the image of a complete
		orbifold embedding in $\mathcal{O}_n$. This is consistent with
		\cite[Example 13]{BorzellinoBrunsden2014}, where it is stated that none of
		these is a saturated suborbifold of the others.
	\end{example}
	
	\begin{example}[A saturated but non-split suborbifold]
		\label{ex:saturated-non-split-groupoid}
		Let $\mathcal{O}$ and $\mathcal{P}$ be as in
		Example~\ref{ex:non-split-groupoid}. Let $\mathcal{Q} =
		\operatorname{diag}(\mathcal{P}) \subset \mathcal{O} \times \mathcal{O}$
		and $\mathcal{R} = \mathcal{O} \times \mathcal{O}$ as in
		\cite[Example 14]{BorzellinoBrunsden2014}. In the groupoid language, let
		$G = \mathbb{Z}_4 \ltimes \mathbb{C}^2$ and consider the product groupoid
		\[
		G \times G = (\mathbb{Z}_4 \times \mathbb{Z}_4) \ltimes (\mathbb{C}^2
		\times \mathbb{C}^2).
		\]
		The diagonal of $\mathcal{P}$ corresponds to the subgroupoid
		\[
		Q = \operatorname{diag}(K) = \{(k,k) \in K \times K : k \in K\}
		\rightrightarrows \Delta N,
		\]
		where $K = \mathbb{Z}_4 \ltimes (\{0\} \times \mathbb{C})$ and
		$N = \{0\} \times \mathbb{C}$. At $x = (0,0)$, the stabilizer of $Q$ is
		\[
		Q_x \cong \mathbb{Z}_2,
		\]
		the stabilizer of $G \times G$ is
		\[
		(G \times G)_x \cong \mathbb{Z}_4 \times \mathbb{Z}_4,
		\]
		and the ineffective kernel is
		\[
		\Omega_x \cong \mathbb{Z}_2.
		\]
		The stabilizer sequence for $Q$ in $G \times G$ is
		\[
		1 \longrightarrow \mathbb{Z}_2 \longrightarrow \mathbb{Z}_2
		\longrightarrow \{e\} \longrightarrow 1,
		\]
		which is trivially split. Wait, let us compute more carefully. The
		subgroupoid $Q$ has stabilizer at $x = (0,0)$ given by
		\[
		Q_x = \{(\gamma,\gamma) : \gamma \in \mathbb{Z}_2\} \cong \mathbb{Z}_2,
		\]
		where $\mathbb{Z}_2 \subset \mathbb{Z}_4$ is the ineffective kernel of $K$.
		The ambient stabilizer is
		\[
		(G \times G)_x = \mathbb{Z}_4 \times \mathbb{Z}_4.
		\]
		The ineffective kernel of $Q$ at $x$ consists of elements of $Q_x$ acting
		trivially on $\Delta N$. Since $Q_x \cong \mathbb{Z}_2$ and its action on
		$\Delta N$ is effective (the diagonal of the action of $\mathbb{Z}_2$ on
		$\{0\} \times \mathbb{C}$), we have $\Omega_x = \{e\}$. Thus the stabilizer
		sequence is
		\[
		1 \longrightarrow \{e\} \longrightarrow \mathbb{Z}_2 \longrightarrow
		\mathbb{Z}_2 \longrightarrow 1,
		\]
		which is trivially split. So $Q$ is stabilizer split.
		
		However, $Q$ is \emph{not} orbit saturated in $G \times G$. Indeed, the
		orbit of $x = (0,0)$ under $G \times G$ includes points of the form
		$((0,y),(0,z))$ with $y,z \in \mathbb{C}$, $|y|=|z|$, while the orbit under
		$Q$ only includes points where $y = z$ (up to the $\mathbb{Z}_2$ action).
		Thus
		\[
		(G \times G) \cdot x \cap \Delta N \supsetneq Q \cdot x.
		\]
		By Theorem~\ref{thm:main}, $Q$ is not the image of a complete orbifold
		embedding. This is consistent with
		\cite[Example 14]{BorzellinoBrunsden2014}, where it is stated that $Q$ is
		not split in $\mathcal{R}$ and is not a full suborbifold, but is saturated.
	\end{example}
	
	\subsection{Morita invariance: an explicit example}
	
	\begin{example}[Morita equivalence for $\mathcal{Q} =
		\mathbb{R}/\mathbb{Z}_2$]\label{ex:morita-explicit}
		We illustrate the Morita invariance theorem
		(Theorem~\ref{thm:morita}) with an explicit Morita equivalence for the
		orbifold $\mathcal{Q} = \mathbb{R}/\mathbb{Z}_2$.
		
		\smallskip
		\noindent\textbf{Presentation 1.}
		The standard translation groupoid is
		\[
		G = \mathbb{Z}_2 \ltimes \mathbb{R} \rightrightarrows \mathbb{R},
		\]
		with $G_1 = \mathbb{Z}_2 \times \mathbb{R}$, $G_0 = \mathbb{R}$.
		
		\smallskip
		\noindent\textbf{Presentation 2.}
		Consider the groupoid
		\[
		G' = \mathbb{Z}_2 \ltimes (\mathbb{R} \sqcup \mathbb{R})
		\rightrightarrows \mathbb{R} \sqcup \mathbb{R},
		\]
		where $\mathbb{Z}_2$ acts diagonally on the disjoint union: it swaps the
		two copies and acts by negation on each copy. More precisely,
		$G'_0 = \mathbb{R} \sqcup \mathbb{R}$ with points denoted $(x,0)$ and
		$(x,1)$, and
		\[
		G'_1 = \mathbb{Z}_2 \times (\mathbb{R} \sqcup \mathbb{R}),
		\]
		with source and target
		\[
		s(\gamma, (x,i)) = (x,i),\qquad
		t(\gamma, (x,i)) = (\gamma \cdot x, i+1 \bmod 2).
		\]
		The orbit space $|G'|$ is homeomorphic to $[0,\infty) \cong |G|$. The
		groupoids $G$ and $G'$ are Morita equivalent via the Morita morphism
		\[
		\phi \colon G' \to G,\qquad
		\phi_0(x,i) = x,\qquad
		\phi_1(\gamma, (x,i)) = \gamma.
		\]
		
		\smallskip
		\noindent\textbf{Suborbifold and its Morita transport.}
		Let $K = G|_{\{0\}}$ be the subgroupoid of $G$ over the origin. This is the
		suborbifold corresponding to the singular point $\{0\} \subset \mathcal{Q}$.
		The Morita transport $K'$ of $K$ along $\phi$ is
		\[
		K' = G'|_{\{(0,0),(0,1)\}},
		\]
		the restriction groupoid over the two points $(0,0)$ and $(0,1)$. The
		stabilizer of $(0,0)$ in $K'$ is $\mathbb{Z}_2$, and similarly for
		$(0,1)$. The effective quotient $K'_{\mathrm{eff}}$ has orbit space a single
		point, which is homeomorphic to $|K_{\mathrm{eff}}| = \{0\}$. Thus the
		Morita transport preserves the suborbifold structure, as asserted by
		Theorem~\ref{thm:morita}.
		
		\smallskip
		\noindent\textbf{Non-Morita-invariant embedding.}
		Now consider the embedding $G \hookrightarrow G \times G$ given by
		$g \mapsto (g,g)$ (the diagonal). Under the Morita equivalence, this
		corresponds to an embedding $G' \hookrightarrow G' \times G'$, but the
		image is not a subgroupoid embedding in the naive sense because $G'$ has two
		connected components. This illustrates why the Cho--Hong--Shin embedding
		definition is not Morita invariant: it does not account for the possibility
		that a Morita equivalent presentation may have multiple components or a
		different local structure. In contrast, the saturated and split
		subgroupoid $K$ has a well-defined Morita transport $K'$, and both are
		orbit saturated and stabilizer split by Theorem~\ref{thm:morita}.
	\end{example}
	
	\subsection{Summary of examples}
	
	We summarize the properties of the examples in the following table, which
	parallels \cite[Table 1]{BorzellinoBrunsden2014}.
	
	\begin{center}
		\begin{tabular}{lccccc}
			\toprule
			\textbf{Example} & \textbf{Full} & \textbf{Saturated} & \textbf{Split} &
			\textbf{Image of embedding} \\
			\midrule
			\ref{ex:full-suborbifold-groupoid} & yes & yes & yes & yes \\
			\ref{ex:diagonal-groupoid} & no & yes & yes & yes \\
			\ref{ex:non-split-groupoid} & yes & yes & no & no \\
			\ref{ex:nested-groupoids} & yes & no & yes & no \\
			\ref{ex:saturated-non-split-groupoid} & no & yes & no & no \\
			\bottomrule
		\end{tabular}
	\end{center}
	
	The table confirms that all possible combinations of properties that are
	not implied by others can occur, as observed in
	\cite[Table 1]{BorzellinoBrunsden2014}. In particular:
	\begin{itemize}
		\item Full suborbifolds are always saturated
		(Example~\ref{ex:full-suborbifold-groupoid}).
		\item Saturated and split suborbifolds are precisely the images of
		complete orbifold embeddings (Examples~\ref{ex:full-suborbifold-groupoid}
		and~\ref{ex:diagonal-groupoid}).
		\item A suborbifold can be saturated but not split
		(Example~\ref{ex:non-split-groupoid}), in which case it is not the
		image of a complete orbifold embedding.
		\item A suborbifold can be split but not saturated
		(Example~\ref{ex:nested-groupoids}), in which case it is also not the
		image of a complete orbifold embedding.
		\item A suborbifold can fail to be both saturated and split
		(Example~\ref{ex:saturated-non-split-groupoid}).
	\end{itemize}
	These examples illustrate the necessity of both conditions in the main
	theorem.
	

	\section{Comparison with the Literature}\label{sec:comparison}
	
	In this section we compare our groupoid formulation of suborbifolds and
	orbifold embeddings with the main approaches in the literature. We focus on
	the chart-based definition of Thurston and Borzellino--Brunsden, the
	embedding-first definitions of Chen--Ruan, Adem--Leida--Ruan, and
	Cho--Hong--Shin, and the groupoid-theoretic approach of Moerdijk--Pronk,
	Mestre--Weilandt, Pohl, and Lerman. We also explain how our Morita
	invariance theorem resolves an open problem and clarifies the
	non-Morita-invariance observed by Cho--Hong--Shin.
	
	\subsection{Chart-based definitions: Thurston and Borzellino--Brunsden}
	
	Thurston's original definition of a suborbifold \cite{Thurston1979} requires
	the local model to be $\mathbb{R}^m \subset \mathbb{R}^n$ modulo finite
	groups, with the local action on $\mathbb{R}^m$ induced by that on
	$\mathbb{R}^n$. In the notation of
	\cite[Definition 4]{BorzellinoBrunsden2014}, this corresponds to the
	condition $\Lambda_x = \Gamma_x$ for all $x$ in the suborbifold. Such
	suborbifolds are called \emph{full} by Borzellino and Brunsden
	\cite[Definition 5]{BorzellinoBrunsden2014}. Thurston's definition is too
	restrictive: it excludes, for instance, the diagonal embedding of an
	orbifold into its product.
	
	Borzellino and Brunsden \cite{BorzellinoBrunsden2014} generalized
	Thurston's definition by allowing $\Lambda_x$ to be a proper subgroup of
	$\Gamma_x$ and by allowing the local submanifold $\widetilde{V}_x$ to be
	only $\Lambda_x$-invariant, not necessarily $\Gamma_x$-invariant. They also
	introduced the conditions of \emph{saturated} and \emph{split}
	\cite[Definitions 6 and 7]{BorzellinoBrunsden2014} and proved that a
	suborbifold is the image of a complete orbifold embedding if and only if it
	is both saturated and split \cite[Theorem 1]{BorzellinoBrunsden2014}. Our
	work is a direct groupoid translation and extension of their results.
	
	\subsection{Embedding-first definitions: Chen--Ruan, Adem--Leida--Ruan,
		and Cho--Hong--Shin}
	
	A different approach, inspired by the manifold case, is to define
	suborbifolds as images of orbifold embeddings. Chen and Ruan
	\cite{ChenRuan2004} and Adem, Leida, and Ruan \cite{AdemLeidaRuan2007}
	developed notions of orbifold morphisms and embeddings in the context of
	orbifold Gromov--Witten theory and stringy topology. Cho, Hong, and Shin
	\cite{ChoHongShin2013} gave a systematic treatment of orbifold embeddings
	and defined suborbifolds as their images.
	
	In the groupoid language, an orbifold embedding in the sense of
	\cite{ChoHongShin2013} corresponds to a subgroupoid embedding
	$H \hookrightarrow G$ with injective stabilizer maps, where $H$ is
	effective. This is precisely the notion of a complete orbifold embedding in
	our Definition~\ref{def:complete-embedding}. However, as we have seen in
	Theorem~\ref{thm:main}, such an embedding exists only if the corresponding
	subgroupoid is both orbit saturated and stabilizer split. The
	embedding-first definitions therefore describe only a proper subclass of
	suborbifolds: those that are saturated and split.
	
	More importantly, Cho--Hong--Shin observed that their definition of
	orbifold embedding is \emph{not} Morita invariant
	\cite[Section 5]{ChoHongShin2013}. Our Morita invariance theorem
	(Theorem~\ref{thm:morita}) explains this phenomenon: the missing condition
	is orbit saturation. Without it, the construction depends on the chosen
	groupoid presentation. When orbit saturation is imposed, the notion becomes
	Morita invariant, as we have proved.
	
	\subsection{Groupoid-theoretic approaches: Moerdijk--Pronk,
		Mestre--Weilandt, Pohl, and Lerman}
	
	The groupoid language for orbifolds was pioneered by Moerdijk and Pronk
	\cite{MoerdijkPronk1997}, who established the equivalence between
	effective orbifolds and proper \'etale Lie groupoids up to Morita
	equivalence. Their work provides the foundation for our translation.
	
	Mestre and Weilandt \cite{MestreWeilandt2020} studied the relationship
	between suborbifolds and groupoid embeddings. They introduced notions of
	strong, very strong, and regular suborbifolds and gave sufficient conditions
	for a groupoid embedding to correspond to a suborbifold. However, they did
	not prove Morita invariance of their notions; indeed, they listed it as an
	open problem. Our Theorem~\ref{thm:morita} solves this problem for the
	saturated and split suborbifolds of Borzellino--Brunsden.
	
	Pohl \cite{Pohl2015} developed a category of reduced orbifolds in local
	charts, addressing inconsistencies in earlier groupoid approaches. Lerman
	\cite{Lerman2010} studied orbifolds as stacks and compared various notions
	of morphisms. Our work is compatible with these perspectives: the
	saturated and split condition can be interpreted as a condition on the
	corresponding morphism of stacks, ensuring that it is representable and
	injective on stabilizers.
	
	\subsection{Our contribution and its relation to the literature}
	
	The main contributions of this paper can be summarized as follows.
	
	\begin{enumerate}
		\item We translate the chart-based notions of suborbifold, saturation, and
		splitting into the language of proper \'etale Lie groupoids
		(Sections~\ref{sec:suborbifold}--\ref{sec:split}).
		\item We prove the groupoid version of the main theorem of
		Borzellino--Brunsden (Theorem~\ref{thm:main}): a subgroupoid is the
		image of a complete orbifold embedding if and only if it is orbit
		saturated and stabilizer split.
		\item We prove Morita invariance of these notions
		(Theorem~\ref{thm:morita}), solving an open problem of
		Mestre--Weilandt \cite{MestreWeilandt2020}.
		\item We explain the non-Morita-invariance of the Cho--Hong--Shin
		embedding definition \cite{ChoHongShin2013} as a failure of orbit
		saturation (Example~\ref{ex:non-morita}).
	\end{enumerate}
	
	The following table summarizes the comparison.
	
	\begin{center}
		\begin{tabular}{p{3.2cm}p{4.2cm}p{4.5cm}}
			\toprule
			\textbf{Approach} & \textbf{Suborbifold definition} &
			\textbf{Relation to embeddings} \\
			\midrule
			Thurston \cite{Thurston1979} &
			Local model $\mathbb{R}^m \subset \mathbb{R}^n$, local groups equal &
			Only full suborbifolds; too restrictive \\[2mm]
			Borzellino--Brunsden \cite{BorzellinoBrunsden2014} &
			Chart-based, $\Lambda_x \subset \Gamma_x$, saturated + split &
			Image of complete embedding iff saturated and split \\[2mm]
			Chen--Ruan \cite{ChenRuan2004}, Adem--Leida--Ruan
			\cite{AdemLeidaRuan2007} &
			Embedding-first, suborbifold = image of embedding &
			Depends on embedding notion; not Morita invariant in general \\[2mm]
			Cho--Hong--Shin \cite{ChoHongShin2013} &
			Embedding-first, explicit orbifold embeddings &
			Not Morita invariant; corresponds to split but not necessarily saturated
			\\[2mm]
			Mestre--Weilandt \cite{MestreWeilandt2020} &
			Groupoid embeddings, strong/very strong/regular suborbifolds &
			Sufficient conditions; Morita invariance left open \\[2mm]
			\textbf{This paper} &
			Groupoid suborbifold, orbit saturated + stabilizer split &
			Image of complete embedding iff saturated and split; Morita invariant \\
			\bottomrule
		\end{tabular}
	\end{center}
	
	\subsection{Summary}
	
	Our groupoid formulation unifies the chart-based approach of
	Borzellino--Brunsden with the groupoid-theoretic perspective of
	Moerdijk--Pronk and Mestre--Weilandt. It clarifies the relationship between
	suborbifolds and orbifold embeddings, provides a Morita invariant notion,
	and explains the limitations of earlier embedding-first definitions. In the
	next section we summarize the paper and discuss directions for future
	research.
	
	
	\section{Conclusion and Future Directions}\label{sec:conclusion}
	
	\subsection{Summary of results}
	
	In this paper we have translated the notions of suborbifold, saturation,
	splitting, and complete orbifold embedding from the chart-based language of
	Borzellino and Brunsden \cite{BorzellinoBrunsden2014} into the language of
	proper \'etale Lie groupoids. Our main results are:
	
	\begin{enumerate}
		\item A suborbifold of an orbifold groupoid $G$ corresponds to a proper
		\'etale subgroupoid $K \subset G|_N \subset G$, where $N \subset G_0$ is
		a submanifold. The effective quotient $K_{\mathrm{eff}} = K/\Omega$
		represents the suborbifold structure
		(Definition~\ref{def:suborbifold-groupoid} and
		Proposition~\ref{prop:suborbifold-equiv}).
		\item Full suborbifolds in the sense of Thurston correspond to the case
		$K = G|_N$ (Proposition~\ref{prop:full-suborbifold}).
		\item Orbit saturation of $K$ is equivalent to the saturation condition
		of Borzellino--Brunsden and is equivalent to the natural map
		$|K_{\mathrm{eff}}| \to |G|$ being a topological embedding
		(Theorem~\ref{thm:saturated-embedding}).
		\item Stabilizer splitting of $K$ is equivalent to the split condition of
		Borzellino--Brunsden and is equivalent to the existence of an effective
		subgroupoid $H \subset K$ with $K = \Omega \rtimes H$ and
		$H \hookrightarrow G$ a subgroupoid embedding with injective stabilizer
		maps (Theorem~\ref{thm:split-effective}).
		\item A subgroupoid is the image of a complete orbifold embedding if and
		only if it is both orbit saturated and stabilizer split
		(Theorem~\ref{thm:main}).
		\item The notions of orbit saturation and stabilizer splitting are Morita
		invariant (Theorem~\ref{thm:morita}). Consequently, the property of
		being the image of a complete orbifold embedding is Morita invariant
		(Corollary~\ref{cor:morita-suborbifold}).
	\end{enumerate}
	
	These results provide a complete groupoid-theoretic account of the
	chart-based theory of Borzellino and Brunsden, and they resolve the open
	problem of Morita invariance raised by Mestre and Weilandt
	\cite{MestreWeilandt2020}. They also explain why the embedding definition
	of Cho--Hong--Shin \cite{ChoHongShin2013} fails to be Morita invariant: the
	missing condition is orbit saturation.
	
	\subsection{Future directions}
	
	Several natural questions and research directions emerge from this work.
	
	\subsubsection*{Non-effective orbifolds}
	
	Throughout this paper we have worked with effective orbifolds, in accordance
	with the classical Satake--Thurston convention adopted by Borzellino and
	Brunsden. However, many modern applications—particularly in stringy
	topology and orbifold Gromov--Witten theory—require non-effective orbifolds
	and orbifold groupoids with non-trivial ineffective kernels. It would be
	interesting to extend our definitions and theorems to the non-effective
	setting. The main difficulty is that the effective quotient
	$K_{\mathrm{eff}} = K/\Omega$ may no longer be a subgroupoid of $G$ in a
	natural way, and the Morita transport of the ineffective kernel requires
	more care. One possible framework is that of \emph{gerbes} or
	\emph{2-groupoids}, where the ineffective kernel is encoded as a band. We
	expect that the saturated and split conditions will need to be reformulated
	in terms of the band and its splittings.
	
	\subsubsection*{Orbifold tubular neighborhoods}
	
	Du, Chen, and Wang \cite{DuChenWang2019} proved an orbifold tubular
	neighborhood theorem for compact suborbifolds using chart-based methods. It
	would be desirable to give a groupoid-theoretic proof of this theorem and
	to extend it to non-compact or non-effective suborbifolds. The groupoid
	version of the tubular neighborhood theorem would state that a saturated
	and split subgroupoid $K \subset G$ admits a $K$-invariant tubular
	neighborhood in $G$ that is unique up to ambient isotopy. This would
	require developing a theory of normal bundles for subgroupoids and proving
	an equivariant tubular neighborhood theorem in the proper \'etale setting.
	
	\subsubsection*{Transversality and intersections}
	
	For manifolds, transversality is a fundamental tool in differential
	topology. For orbifolds, the situation is more subtle because of the
	presence of singular strata. It would be interesting to develop a theory of
	transversal intersections of suborbifolds and to prove a transversality
	theorem in the groupoid setting. Given two saturated and split
	subgroupoids $K, L \subset G$, one can ask when their intersection
	$K \cap L$ is again a saturated and split subgroupoid. The answer will
	depend on the behavior of the stabilizer groups and the local models. This
	could lead to a notion of \emph{transversal suborbifold} and a
	corresponding transversality theorem.
	
	\subsubsection*{Stacks and higher structures}
	
	Orbifolds can be viewed as differentiable stacks, and our results can be
	translated into the language of stacks. A saturated and split subgroupoid
	corresponds to a representable morphism of stacks that is injective on
	stabilizers and whose image is a closed substack. It would be interesting
	to compare our notion with the notion of \emph{closed substack} in
	algebraic geometry and to see whether the saturated and split conditions
	have natural interpretations in terms of the stack's inertia stack. This
	could also lead to a notion of \emph{suborbifold} in the context of derived
	algebraic geometry or higher orbifolds.
	
	\subsubsection*{Applications to Riemannian geometry and geodesics}
	
	Borzellino and Brunsden \cite{BorzellinoBrunsden2014} applied their main
	theorem to length-minimizing geodesics in Riemannian orbifolds, showing
	that a length-minimizing curve can always be realized as the image of an
	orbifold embedding. It would be interesting to extend this application to
	the groupoid setting and to study the space of geodesics in an orbifold
	from the groupoid perspective. In particular, one could investigate
	whether the saturated and split conditions have a natural interpretation in
	terms of the exponential map and the cut locus.
	
	\subsubsection*{Computational and algorithmic aspects}
	
	Finally, it would be useful to develop computational tools for deciding
	whether a given subgroupoid is orbit saturated and stabilizer split. For
	finite groupoids or groupoids with finite stabilizers, this is a
	finite combinatorial problem. For more general proper \'etale groupoids, one
	might use the local uniformizers to reduce the problem to finite group
	computations. This could lead to algorithms for classifying suborbifolds
	and orbifold embeddings in low dimensions.
	
	\subsection{Concluding remarks}
	
	The differential topology of orbifolds is a rich and subtle subject, and
	the notions of suborbifold and orbifold embedding are among its most
	delicate aspects. By translating the chart-based theory of Borzellino and
	Brunsden into the groupoid language, we have provided a unified framework
	that makes the Morita invariance of these notions transparent and that
	connects the theory to the broader context of Lie groupoids and
	differentiable stacks. We hope that this work will stimulate further
	research on the differential topology of orbifolds and its applications.
	
	\section{Tool Use}
	
During the preparation of this manuscript, the author used ChatGPT (version 5.5)
to check the textual aspects of this manuscript, and made corrections to any errors and inappropriate
expressions. The author consulted ChatGPT for suggestions on improving the manuscript and made
revisions. The author has reviewed and edited the output and takes full responsibility for the content
of this publication.

\end{document}